\documentclass{krm}
\usepackage{amsmath}
  \usepackage{paralist}
  \usepackage{graphics} %% add this and next lines if pictures should be in esp format
  \usepackage{epsfig} %For pictures: screened artwork should be set up with an 85 or 100 line screen
 \usepackage[colorlinks=true]{hyperref}
    \hypersetup{urlcolor=blue, citecolor=red}

\def\currentvolume{3}
 \def\currentissue{4}
  \def\currentyear{2010}
   \def\currentmonth{December}
    \def\ppages{685--728}
     \def\DOI{10.3934/krm.2010.3.685}

\def\v{\varepsilon}

\def\x{\xi}
\def\t{\theta}
\def\T{\Theta}
\def\k{\kappa}

\def\a{\alpha}
\def\b{\beta}
\def\g{\gamma}
\def\d{\delta}
\def\l{\lambda}
\def\mb{\mathbf}
\def\f{\frac}
\def\p{\phi}

\def\ra{\rightarrow}
\def\s{\sigma}
\def\z{\zeta}

\def\di{\displaystyle}
\def\i{\infty}

\newtheorem{theorem}{Theorem}[section]

\newtheorem{lemma}[theorem]{Lemma}

\theoremstyle{definition}

\newtheorem{remark}{Remark}

\title[Fluid Dynamic Limit to Euler Equations]
      {Fluid Dynamic Limit to the Riemann Solutions of  Euler
      Equations: I. Superposition of rarefaction waves and contact discontinuity}

\author[Feimin Huang, Yi Wang and Tong Yang]{}


 \email{fhuang@amt.ac.cn;wangyi@amss.ac.cn;matyang@cityu.edu.hk}
\begin{document}
\maketitle

% Enter the first author's name and address:
\centerline{\scshape Feimin Huang and Yi Wang}
\medskip
{\footnotesize
% please put the address of the first author
 \centerline{Institute of Applied Mathematics, AMSS and}
   \centerline{Hua Loo-Keng Key Laboratory of
 Mathematics, Academia Sinica}
 \centerline{ Beijing 100190, China}
 } % Do not forget to end the {\footnotesize by the sign }

\medskip
\centerline{\scshape Tong Yang}
\medskip
{\footnotesize
 % please put the address of the second  and third author
 \centerline{ Department of Mathematics, City University of
HongKong}\centerline{ HongKong, China}

}
\bigskip

%% The name of the associate editor will be entered by an editorial staff
 \centerline{(Communicated by Seiji Ukai)}

%The abstract of your paper
\begin{abstract}
 Fluid dynamic limit to compressible Euler equations
from compressible Navier-Stokes equations and Boltzmann equation has
been an active topic   with limited success so far. In this paper,
we consider the case when the solution of the Euler equations is a
Riemann solution consisting two rarefaction waves and a contact
discontinuity and prove this limit for both Navier-Stokes equations
and the Boltzmann equation when the viscosity, heat conductivity
coefficients and the Knudsen number tend to zero respectively. In
addition,  the uniform convergence rates in terms of the above
physical parameters are also obtained. It is noted that this is the
first rigorous proof of this limit for a Riemann solution with
superposition of three waves even though the fluid dynamic limit for
a single wave has been proved.
\end{abstract}

%The title of your section 1
\tableofcontents

\section{Introduction}
\renewcommand{\theequation}{\arabic{section}.\arabic{equation}}
\setcounter{equation}{0}

This paper concerns the fluid dynamic limit to the compressible
Euler equations for two physical models, that is, the compressible
Navier-Stokes equations and the Boltzmann equation. In the first
part, we consider zero dissipation limit of
 the compressible Navier-Stokes system for viscous and heat
conductive fluid in the Lagrangian coordinates:
\begin{equation}
\left\{
\begin{array}{ll}
\displaystyle v_{t}-u_{x}=0, \\
\displaystyle u_{t}+p_{x}=\v(\frac {u_{x}}v)_{x},\\
\displaystyle\bigl(e+\frac{u^{2}}{2}\bigr)_{t}+
(pu)_{x}=(\kappa\frac{\theta_{x}}{v}+\v \frac{uu_{x}}{v})_{x},
\end{array}
\right.
 \label{(1.1)}
\end{equation}
where the functions $v(t,x)>0,u(t,x),\t(t,x)>0$ represent the
specific volume, velocity and the absolute temperature of the gas
respectively. And $p=p(v,\t)$ is the pressure, $e=e(v,\t)$ is the
internal energy, $\v>0$ is the viscosity coefficient,  $\k>0$ is the
coefficient of the heat conductivity. Here, both $\v$ and $\k$ are
taken as  positive constants. And we consider the perfect gas where
\begin{equation}
p=\f{R\t}{v}=A v^{ -\gamma }\exp \big(\frac{\g-1}{ R}s \big),\qquad
e=\f{R\t}{\g-1}, \label{(1.2)}
\end{equation}
with $s$ denoting the entropy of the gas and $A,R>0$ , $\g>1$ being
the gas parameters.

 Formally, as the coefficients $\k$ and $\v$ tend to zero, the
limiting system of  \eqref{(1.1)} is the compressible Euler
equations
\begin{equation}
\left\{
\begin{array}{ll}
v_t-u_x=0,\\
u_t+p_x=0,\\
(e+\f{u^2}{2})_t+(pu)_x=0.
\end{array}
\right.\label{(1.2+)}
\end{equation}

The study of this limiting process of viscous flows when the
viscosity and heat conductivity coefficients tend to zero, is one of
the important problems in the theory of the compressible fluid.
 When the solution of the inviscid flow is smooth, the zero
dissipation limit  can be solved by classical scaling method.
However, the inviscid compressible flow usually contains
discontinuities, such as shock waves and contact discontinuities.
Therefore, how to justify the zero dissipation limit to the Euler
equations with basic wave patterns is a natural and difficult
problem.

Keeping in mind that the Navier-Stokes equations can be derived from
the Boltzmann equation through the Chapman-Enskog expansion
 when the Knudsen number is close
to zero, we assume the following condition on the viscosity constant
$\v$ and the heat conductivity coefficient $\k$ in the system
\eqref{(1.1)}, cf. also \cite{Jiang-Ni-Sun}:
\begin{equation}
\left\{
\begin{array}{l}
\di \k=O(\v)\qquad \qquad \rm as \qquad\v\rightarrow0;\\
\di \nu\doteq\f{\k(\v)}{\v}\geq c>0\qquad {\rm for ~some ~positive~
constant}~ c,\quad \rm as \quad\v\rightarrow0.
\end{array}
\right.  \label{(1.3)}
\end{equation}
%Note that the above assumptions are reasonable in the sense that
%when the Chapman-Enskog expansion for the Boltzmann equation is
%done, the viscosity coefficient and the heat conducting coefficient
%do satisfy the relation \eqref{(1.3)}, although they both are smooth
%function of the temperature, see the details in the following
%introduction for Boltzmann equation.

Now we briefly review some recent results on the zero dissipation
limit of the compressible fluid with basic wave patterns. For the
hyperbolic conservation laws with artificial viscosity
$$
u_t+f(u)_x=\v u_{xx},
$$
Goodman-Xin \cite{Goodman-Xin} verified the viscous limit for
piecewise smooth solutions separated by non-interacting shock waves
using a matched asymptotic expansion method. For the compressible
isentropic Navier-Stokes equations, Hoff-Liu \cite{Hoff-Liu} first
proved the vanishing viscosity limit for  piecewise constant
solutions separated by non-interacting shocks even with initial
layer. Later Xin \cite{Xin1} obtained the zero dissipation limit for
rarefaction waves and Wang \cite{Wang-H} generalized the result of
Goodmann-Xin \cite{Goodman-Xin} to the isentropic Navier-Stokes
equations.

For the inviscid limit of the full compressible Navier-Stokes
equations \eqref{(1.1)}, Jiang-Ni-Sun \cite{Jiang-Ni-Sun} justified
the zero dissipation limit of the system \eqref{(1.1)} for centered
rarefaction waves. Wang \cite{Wang} proved the zero dissipation
limit of the system \eqref{(1.1)} for  piecewise smooth solutions
separated by shocks using the matched asymptotic expansion method
introduced in \cite{Goodman-Xin}. Recently, Xin-Zeng \cite{Xin-Zeng}
considered the zero dissipation limit of the system \eqref{(1.1)}
for single rarefaction wave with well prepared initial data and
obtained a uniform decay rate in terms of  the dissipation
coefficients. And Ma \cite{Ma} obtained the zero dissipation limit
of a single strong contact discontinuity in any fixed time interval
with a decay rate.

However, to our knowledge, so far there is no result on the zero
dissipation limit of the system \eqref{(1.1)} for superposition of
different types of basic wave patterns. In the first part of this
paper, we investigate the fluid dynamic limit of the compressible
Navier-Stokes equations  when the corresponding Euler equations have
the Riemann solution as a superposition of two rarefaction waves and
a contact discontinuity. For this, we need to study the interaction
between the rarefaction waves and contact discontinuity.

In the second part of the paper, we study the hydrodynamic limit of
the Boltzmann equation \cite{Boltzmann} with
 slab symmetry
\begin{equation}
f_t + \xi_1f_x= \f{1}{\v}Q(f,f),~(f,t,x,\xi)\in {\mb R}\times{\mb
R}^+\times {\mb R}
 \times {\mb R}^3, \label{(1.4)}
\end{equation}
where $\xi=(\xi_1,\xi_2,\xi_3)\in {\mb R}^3$, $f(t,x,\xi)$ is the
 density distribution function of  particles at time $t$ with location
 $x$  and velocity $\xi$,  and $\v>0$ is called  the Knudsen number
which is proportional to the mean free path. Remark that the
notation $\v$ here is same as the viscosity of the compressible
Navier-Stokes equations (\ref{(1.1)}), but it has different physical
meanings from (\ref{(1.1)}) in different equations and related
contexts.
%The equation \eqref{(1.4)} was established by Boltzmann
%\cite{Boltzmann} in 1872 to describe the motion of  rarefied gases
%and it is a fundamental equation in statistics physics.

For monatomic gas, the rotational invariance of the particles leads
to the following  bilinear form for the collision operator
$$
\begin{array}{ll}
 \di Q(f,g)(\xi) = \f{1}{2}
\int_{{\mb R}^3}\!\!\int_{{\mb S}_+^2} \Big(f(\xi')g(\xi_*')+
f(\xi_*')g(\xi')-f(\xi) g(\xi_*) - f(\xi_*)g(\xi)\Big)\\
\di \hspace{8cm}\qquad B(|\xi-\xi_*|, \hat\t)
 \; d \xi_* d\Gamma,
 \end{array}
$$
where $\xi',\xi_*'$ are the velocities after an elastic collision of
two particles with velocities  $\xi,\xi_*$ before the collision.
Here, $\hat\t$ is the angle between the relative velocity
$\xi-\xi_*$ and the unit vector $\Gamma$ in ${\mb S}^2_+=\{\Gamma\in
{\mb S}^2:\ (\xi-\xi_*)\cdot \Gamma\geq 0\}$. The conservation of
momentum and energy gives the following relation between the
velocities before and after collision:
$$
\left\{
\begin{array}{l}
 \xi'= \xi -[(\xi-\xi_*)\cdot \Gamma] \; \Gamma, \\[3mm]
 \xi_*'= \xi_* + [(\xi-\xi_*)\cdot \Gamma] \; \Gamma.
\end{array}
\right.
$$

In this paper, we consider the Boltzmann equation for two basic
models, that is, the hard sphere model and the hard potential
including Maxwellian molecules under the assumption of angular
cut-off. For this, we assume that the collision kernel
$B(|\xi-\xi_*|,\hat\t)$ takes one of the following two forms,
$$
B(|\xi-\xi_*|,\hat\t)=|(\xi-\xi_*, \Gamma)|=|\x-\x_*|\cos\hat\t,
$$
and
$$
B(|\xi-\xi_*|,\hat\t)=|\xi-\xi_*|^{\f{n-5}{n-1}}b(\hat\t),\quad
b(\hat\t)\in L^1([0, \pi]),~n\ge 5.
$$
 Here, $n$ is the index in the potential of
 inverse power law which is
proportional to $r^{1-n}$ with $r$ being the distance between two
concerned particles.

Formally,  when the Knudsen number $\v$ tends to zero, the limit  of
the Boltzmann equation \eqref{(1.4)} is
 the classical system of Euler equations
\begin{equation}
\left\{
\begin{array}{l}
\di\rho_t+(\rho u_1)_x=0,\\
\di(\rho u_1)_t+(\rho u_1^2+p)_x=0,\\
\di(\rho u_i)_t+(\rho u_1u_i)_x=0,~i=2,3,\\
\di[\rho(E+\f{|u|^2}{2})]_t+[\rho u_1(E+\f{|u|^2}{2})+pu_1]_x=0,
\end{array}
\right. \label{(1.5)}
\end{equation}
where
\begin{equation}
\left\{
\begin{array}{l}
\di\rho(t,x)=\int_{\mb{R}^3}\varphi_0(\xi)f(t,x,\xi)d\xi,\\
\di\rho u_i(t,x)=\int_{\mb{R}^3}\varphi_i(\xi)f(t,x,\xi)d\xi,~i=1,2,3,\\
\di\rho(E+\f{|u|^2}{2})(t,x)=\int_{\mb{R}^3}\varphi_4(\xi)f(t,x,\xi)d\xi.
\end{array}
\right. \label{(1.6)}
\end{equation}
Here, $\rho$ is the density, $u=(u_1,u_2,u_3)$ is the macroscopic
velocity, $E$ is the internal energy of the gas, and $p=R\rho\t$
with $R$ being the gas constant is  the pressure. Note that the
temperature $\t$ is related to the internal energy by
$E=\f{3}{2}R\t$, and $\varphi_i(\xi)(i=0,1,2,3,4)$ are the collision
invariants given by
$$
\left\{
\begin{array}{l}
 \varphi_0(\xi) = 1, \\
 \varphi_i(\xi) = \xi_i \ \ {\textrm {for} }\ \  i=1,2,3, \\
 \varphi_4(\xi) = \f{1}{2} |\xi|^2,
\end{array}
\right.
$$
that satisfy
$$
\int_{{\mb R}^3} \varphi_i(\xi) Q(h,g) d \xi =0,\quad {\textrm {for}
} \ \  i=0,1,2,3,4.
$$

How to justify the above limit, that is,  the Euler equation
(\ref{(1.5)}) from Boltzmann equation (\ref{(1.4)}) when Knudsen
number $\v$ tends to zero is an open problem going way back to the
time of Maxwell. For this,  Hilbert  introduced the famous Hilbert
expansion to show  formally that  the first order approximation of
the Boltzmann equation gives  the Euler equations. On the other
hand, it is important to verify this limit process rigorously in
mathematics. For the case when the Euler equation has smooth
solutions, the zero Knudsen number limit of the Boltzmann equation
has been studied even in the case with an initial layer, cf.
Ukai-Asano \cite{Asona-Ukai}, Caflish \cite{Caflish}, Lachowicz
\cite{Lachowicz} and Nishida \cite{Nishida} etc. However, as is
well-known, solutions of the Euler equations (\ref{(1.5)}) in
general develop  singularities, such as shock waves and contact
discontinuities. Therefore,  how to verify the fluid limit from
Boltzmann equation to the Euler equations with basic wave patterns
becomes an natural problem. In this direction,  Yu \cite{Yu} showed
that when the solution of the Euler equations (\ref{(1.5)}) contains
only non-interacting shocks, there exists a sequence of solutions to
 the Boltzmann equation that converge to a local Maxwellian defined by the
 solution of the Euler equations (\ref{(1.5)}) uniformly
away from the shock in any fixed time interval. In this work, the
inner and outer expansions developed by Goodman-Xin
\cite{Goodman-Xin} for conservation laws and the Hilbert expansion
were skillfully and cleverly used. Recently, Huang-Wang-Yang
 \cite{Huang-Wang-Yang} proved the fluid dynamic limit of
the Boltzmann equation to the Euler equations for a single contact
discontinuity where the uniform decay rate was also obtained. And
Xin-Zeng \cite{Xin-Zeng} proved the fluid dynamic limit of the
compressible Navier-Stokes equations and Boltzmann equation to the
Euler equations with non-interacting rarefaction waves. About the
detailed introductions of  the Boltzmann equation and its
hydrodynamic limit, see the books
\cite{Cercignani-Illner-Pulvirenti}, \cite{Esposito} etc.

In this paper, we will study the hydrodynamic limit of the Boltzmann
equation when the corresponding Euler equations have a Riemann
solution as a superposition of two rarefaction waves and a contact
discontinuity. More precisely, given a Riemann solution of the Euler
equations (\ref{(1.5)}) with superposition  of two rarefaction waves
and  a contact discontinuity, we will show that there exists a
family of solutions to the Boltzmann equation that converge to a
local Maxwellian defined by the Euler solution
 uniformly away from the contact
discontinuity for strictly positive time as $\v\rightarrow 0$.
Moreover, a uniform convergence rate in $\v$ is also given.
%The proof is
%obtained by a scaling transformation of the independent variables
%and the perturbation together with  the energy method introduced
% by Liu-Yang-Yu\cite{Liu-Yang-Yu}.

As mentioned above for the compressible Navier-Stokes equations, we
also need to study the detailed wave interactions through this
limiting process.
%the new
%difficulties occur when treating the interaction terms of the
%rarefaction waves and contact discontinuities. So we introduce a
%suitable lift $t_0=\v^{\f15}$ in the construction of the approximate
%waves to overcome this new difficultiy in the hydrodynamic limit.

For later use, we now  briefly present the micro-macro decomposition
around the local Maxwellian defined by the solution to the Boltzmann
equation, cf. \cite{Liu-Yang-Yu} and \cite{Liu-Yu}. For a solution
$f(t,x,\xi)$ of the Boltzmann equation (\ref{(1.4)}), set
$$
f(t,x,\xi)=\mb{M}(t,x,\xi)+\mb{G}(t,x,\xi),
$$
where the local Maxwellian
$\mb{M}(t,x,\xi)=\mb{M}_{[\rho,u,\t]}(\xi)$ represents the
macroscopic (fluid) component of the solution, which is naturally
defined by the five conserved quantities, i.e., the mass density
$\rho(t,x)$, the momentum $\rho u(t,x)$, and the total energy
$\rho(E+\f{1}{2}|u|^2)(t,x)$ in (\ref{(1.6)}), through
\begin{equation}
\mb{M}=\mb{M}_{[\rho,u,\t]} (t,x,\xi) = \f{\rho(t,x)}{\sqrt{ (2 \pi
R \t(t,x))^3}} e^{-\f{|\xi-u(t,x)|^2}{2R\t(t,x)}}. \label{(1.7)}
\end{equation}
And $\mb{G}(t,x,\xi)$ being the difference between the solution and
the above local\break Maxwellian represents the microscopic
(non-fluid) component.

For convenience, we denote the inner product of $h$ and $g$ in
$L^2_{\xi}({\mb R}^3)$ with respect to a given Maxwellian
$\tilde{\mb{M}}$ by:
$$
 \langle h,g\rangle_{\tilde{\mb{M}}}\equiv \int_{{\mb R}^3}
 \f{1}{\tilde{\mb{M}}}h(\xi)g(\xi)d \xi.
$$
 If $\tilde{\mb{M}}$ is the local
Maxwellian $\mb{M}$ defined in (\ref{(1.7)}), with respect to the
corresponding inner product, the macroscopic space is spanned by the
following five pairwise orthogonal base
$$
\left\{
\begin{array}{l}
 \chi_0(\xi) \equiv {\di\f1{\sqrt{\rho}}\mb{M}}, \\[2mm]
 \chi_i(\xi) \equiv {\di\f{\xi_i-u_i}{\sqrt{R\t\rho}}\mb{M}} \ \ {\textrm {for} }\ \  i=1,2,3, \\[2mm]
 \chi_4(\xi) \equiv
 {\di\f{1}{\sqrt{6\rho}}(\f{|\xi-u|^2}{R\t}-3)\mb{M}},\\
 \langle\chi_i,\chi_j\rangle=\delta_{ij}, ~i,j=0,1,2,3,4.
 \end{array}
\right.
$$
In the following, if $\tilde{\mb{M}}$ is the local Maxwellian
$\mb{M}$, we just use the simplified notation
$\langle\cdot,\cdot\rangle$ to denote
 the
inner product $\langle\cdot,\cdot\rangle_{\mb{M}}$. The macroscopic
projection $\mb{P}_0$ and microscopic projection $\mb{P}_1$ can be
defined as follows
$$
\left\{
\begin{array}{l}
 \mb{P}_0h = {\di\sum_{j=0}^4\langle h,\chi_j\rangle\chi_j,} \\
 \mb{P}_1h= h-\mb{P}_0h.
 \end{array}
\right.
$$
The projections $\mb{P}_0$ and $\mb{P}_1$ are orthogonal
and satisfy
$$
\mb{P}_0\mb{P}_0=\mb{P}_0, \mb{P}_1\mb{P}_1=\mb{P}_1,
\mb{P}_0\mb{P}_1=\mb{P}_1\mb{P}_0=0.
$$
Note that a function $h(\xi)$ is called microscopic or non-fluid if
$$
\int h(\xi)\varphi_i(\xi)d\xi=0,~i=0,1,2,3,4,
$$
where $\varphi_i(\xi)$ is the collision invariants.

 Under the above micro-macro decomposition, the solution $f(t,x,\xi)$ of the Boltzmann
equation (\ref{(1.4)}) satisfies
$$
\mb{P}_0f=\mb{M},~\mb{P}_1f=\mb{G},
$$
and the Boltzmann equation (\ref{(1.4)}) becomes
\begin{equation}
(\mb{M}+\mb{G})_t+\xi_1(\mb{M}+\mb{G})_x
=\f{1}{\v}[2Q(\mb{M},\mb{G})+Q(\mb{G},\mb{G})]. \label{(1.9)}
\end{equation}
By multiplying the equation (\ref{(1.9)}) by the collision
invariants $\varphi_i(\xi)(i=0,1,2,3,4)$ and integrating the
resulting equations with respect to $\xi$ over ${\mb R}^3$, one has
the following fluid-type system for the fluid components:
\begin{equation}
\left\{
\begin{array}{lll}
\di \rho_{t}+(\rho u_1)_x=0, \\
\di (\rho u_1)_t+(\rho u_1^2
+p)_x=-\int\xi_1^2\mb{G}_xd\xi,  \\
\di (\rho u_i)_t+(\rho u_1u_i)_x=-\int\xi_1\xi_i\mb{G}_xd\xi,~ i=2,3,\\
\di [\rho(E+\f{|u|^2}{2})]_t+[\rho
u_1(E+\f{|u|^2}{2})+pu_1]_x=-\int\f12\xi_1|\xi|^2\mb{G}_xd\xi.
\end{array}
\right. \label{(1.10)}
\end{equation}

 Note that the above fluid-type system is not
closed and one more  equation for the non-fluid component ${\mb{G}}$
is needed and it
 can be obtained by applying the projection operator
$\mb{P}_1$ to the equation (\ref{(1.9)}):
\begin{equation}
\mb{G}_t+\mb{P}_1(\xi_1\mb{M}_x)+\mb{P}_1(\xi_1\mb{G}_x)
=\f{1}{\v}\left[\mb{L}_\mb{M}\mb{G}+Q(\mb{G}, \mb{G})\right].
\label{(1.11)}
\end{equation}
Here $\mb{L}_\mb{M}$ is the linearized collision operator of
$Q(f,f)$
 with respect to the local Maxwellian $\mb{M}$:
$$
\mb{L}_\mb{M} h=2Q(\mb{M}, h)=Q(\mb{M}, h)+ Q(h,\mb{M}).
$$
Note that  the null space $\mathfrak{N}$ of $\mb{L}_\mb{M}$ is
spanned by the macroscopic variables:
$$
\chi_j(\xi), ~j=0,1,2,3,4.
$$
Furthermore, there exists a positive constant $\sigma_0>0$ such that
for any function $h(\xi)\in \mathfrak{N}^\bot$, cf. \cite{Grad},
$$
<h,\mb{L}_\mb{M}h>\le -\sigma_0<\nu(|\xi|)h,h>,
$$
where $\nu(|\xi|)$ is the collision frequency. For the hard sphere
model and the hard potential including Maxwellian molecules with
angular cut-off, the collision frequency $\nu(|\xi|)$ has the
following property
$$
0<\nu_0<\nu(|\xi|)\le c(1+|\xi|)^{\kappa_0},
$$
for some positive constants $\nu_0, c$ and $0\le\kappa_0\le 1$.

Consequently, the linearized collision operator $\mb{L}_\mb{M}$ is a
dissipative operator on $L^2({\mb R}^3)$, and its inverse
$\mb{L}_\mb{M}^{-1}$  exists in $\mathfrak{N}^\bot$.

It follows from (\ref{(1.11)}) that
\begin{equation}
\mb{G}=\v \mb{L}_\mb{M}^{-1}[\mb{P}_1(\xi_1\mb{M}_x)] +\Pi,
\label{(1.12)}
\end{equation}
with
\begin{equation}
\Pi=\mb{L}_\mb{M}^{-1}[\v(\mb{G}_t+\mb
{P}_1(\xi_1\mb{G}_x))-Q(\mb{G}, \mb{G})].\label{(1.13)}
\end{equation}

Plugging the equation (\ref{(1.12)}) into (\ref{(1.10)}) gives
\begin{equation}
\left\{
\begin{array}{l}
\di \rho_{t}+(\rho u_1)_x=0,\\
\di (\rho u_1)_t+(\rho u_1^2 +p)_x=\f{4\v}{3}(\mu(\t)
u_{1x})_x-\int\xi_1^2\Pi_xd\xi,  \\
\di (\rho u_i)_t+(\rho u_1u_i)_x=\v(\mu(\t)
u_{ix})_x-\int\xi_1\xi_i\Pi_xd\xi,~ i=2,3,\\
\di [\rho(E+\f{|u|^2}{2})]_t+[\rho
u_1(E+\f{|u|^2}{2})+pu_1]_x=\v(\lambda(\t)\t_x)_x+\f{4\v}{3}(\mu(\t)u_1u_{1x})_x\\
\di\qquad\qquad +\v\sum_{i=2}^3(\mu(\t)u_iu_{ix})_x
-\int\f12\xi_1|\xi|^2\Pi_xd\xi,
\end{array}
\right. \label{(1.14)}
\end{equation}
where the viscosity coefficient $\mu(\t)>0$ and the heat
conductivity coefficient $\lambda(\t)>0$ are smooth functions of the
temperature $\t$. Here,  we normalize the gas constant $R$ to be
$\f{2}{3}$ so that $E=\t$ and $p=\f23\rho\t$. The explicit formula
of $\mu(\t)$ and $\lambda(\t)$ can be found for example in \cite{CC}, we omit it here for brevity.

Since the problem considered in this paper is  one dimensional in
the space variable $x\in {\bf R}$, in the macroscopic level, it is
more convenient to rewrite the equation (\ref{(1.4)}) and the system
(\ref{(1.5)}) in the {\it Lagrangian} coordinates as in the study of
conservation laws. That is, set the coordinate transformation:
$$
 x\Rightarrow \int_0^x \rho(t,y)dy, \qquad t\Rightarrow t.
$$
We will still denote the {\it Lagrangian} coordinates by $(t,x)$ for
simplicity of notation. Then (\ref{(1.4)}) and (\ref{(1.5)}) in the
Lagrangian coordinates become, respectively,
\begin{equation}
f_t-\f{u_1}{v}f_x+\f{\xi_1}{v}f_x=\f{1}{\v}Q(f,f),\label{(1.15)}
\end{equation}
and
\begin{equation}
\left\{
\begin{array}{llll}
\di v_{t}-u_{1x}=0,\\
\di u_{1t}+p_x=0,\\
\di u_{it}=0, ~i=2,3,\\
\di (\t+\f{|u|^{2}}{2}\bigr)_t+ (pu_1)_x=0.\\
\end{array}
\right.\label{(1.16)}
\end{equation}
Also, (\ref{(1.10)})-(\ref{(1.14)}) take the form
\begin{equation}
\left\{
\begin{array}{llll}
\di v_t-u_{1x}=0,\\
\di u_{1t}+p_x=-\int\xi_1^2\mb{G}_xd\xi,\\
\di u_{it}=-\int\xi_1\xi_i\mb{G}_xd\xi,
~i=2,3,\\
\di\bigl(\t+\f{|u|^{2}}{2}\bigr)_{t}+ (pu_1)_x=-\int\f12\xi_1|\xi|^2\mb{G}_xd\xi,\\
\end{array}
\right.\label{(1.17)}
\end{equation}
\begin{equation}
\mb{G}_t-\f{u_1}{v}\mb{G}_x+\f{1}{v}\mb{P}_1(\xi_1\mb{M}_x)+\f{1}{v}\mb{P}_1(\xi_1\mb{G}_x)=\f{1}{\v}(\mb{L}_\mb{M}\mb{G}+Q(\mb{G},\mb{G})),\label{(1.18)}
\end{equation}
with
\begin{equation}
\mb{G}=\v \mb{L}^{-1}_\mb{M}(\f{1}{v} \mb{P}_1(\xi_1
\mb{M}_x))+\Pi_1,\label{(1.19)}
\end{equation}
\begin{equation}
\Pi_1=\mb{L}_\mb{M}^{-1}[\v(\mb{G}_t-\f{u_1}v\mb{G}_x+\f{1}{v}\mb{P}_1(\xi_1\mb{G}_x))-Q(\mb{G},\mb{G})].\label{(1.21)}
\end{equation}
and
\begin{equation}
\left\{
\begin{array}{llll}
\di v_t-u_{1x}=0,\\
\di u_{1t}+p_x=\f{4\v}{3}(\f {\mu(\t)}vu_{1x})_{x}-\int\xi_1^2\Pi_{1x}d\xi,\\
\di u_{it}=\v(\f{\mu(\t)}{v}u_{ix})_x-\int\xi_1\xi_i\Pi_{1x}d\xi,
~i=2,3,\\
\di\bigl(\t+\f{|u|^{2}}{2}\bigr)_{t}+
(pu_1)_{x}=\v(\f{\lambda(\t)}{v}\t_x)_x+\f{4\v}{3}(\f{\mu(\t)}{v}u_1u_{1x})_x\\
\di\qquad+\v\sum_{i=2}^3(\f{\mu(\t)}{v} u_iu_{ix})_x
-\int\f12\xi_1|\xi|^2\Pi_{1x}d\xi.
\end{array}
\right. \label{(1.22)}
\end{equation}

With the above preparation, the main results in this paper for both
the compressible Navier-Stokes equations and the Boltzmann equation
will be given in the next section. And the proof of the zero
dissipation limit for the compressible Navier-Stokes equations will
be given in Section 3 while the proof of hydrodynamic limit for the
Boltzmann equation will be given in the last section.

\section{Main results}
\setcounter{equation}{0}

\subsection{Compressible Navier-Stokes equations}

It is well known that for the Euler equations, there are three basic
wave patterns, shock, rarefaction wave and contact discontinuity.
And the Riemann solution to the Euler equations has a basic wave
pattern consisting the superposition of these three waves with the
contact discontinuity in the middle. For later use, let us firstly
recall the wave curves for the two types of basic waves studied in
this paper.

Given the right end state $(v_+,u_+,\t_+)$,  the following wave
curves in the phase space $(v,u,\t)$ are defined with $v>0$ and
$\t>0$ for the Euler equations.

$\bullet$ Contact discontinuity wave curve:
\begin{equation}
CD(v_+,u_+,\t_+)= \{(v,u,\t)  |  u=u_+, p=p_+, v \not\equiv v_+
 \}. \label{(2.1)}
\end{equation}

$\bullet$ $i$-Rarefaction wave curve $(i=1,3)$:
\begin{equation}
 R_i (v_+, u_+, \theta_+):=\Bigg{ \{} (v, u, \theta)\Bigg{ |}v<v_+ ,~u=u_+-\int^v_{v_+}
 \lambda_i(\eta,
s_+) \,d\eta,~ s(v, \theta)=s_+\Bigg{ \}}\label{(2.2)}
\end{equation}
where $s_+=s(v_+,\t_+)$ and $\l_i=\l_i(v,s)$ is $i$-th
characteristic speed of the Euler system \eqref{(1.2+)} or
\eqref{(1.16)}.

Accordingly, when we study the Navier-Stokes equations, the
corresponding wave profiles can be defined approximately as follows,
cf. \cite{Huang-Xin-Yang}, \cite{Xin1}.

\subsubsection{Contact discontinuity}

If $(v_-,u_-,\t_-)\in CD(v_+,u_+,\t_+)$, i.e.,
$$
u_-=u_+,~p_-=p_+,
$$
then the following Riemann problem of the Euler system
\eqref{(1.2+)} with Riemann initial data
$$
(v,u,\t)(t=0,x)=\left\{
\begin{array}{ll}
(v_-,u_-,\t_-),\qquad & x<0,\\
(v_+,u_+,\t_+),\qquad & x>0
\end{array}
\right.
$$
admits a single contact discontinuity solution
\begin{equation}
(v^{cd},u^{cd},\t^{cd})(t,x)=\left\{
\begin{array}{ll}
(v_-,u_-,\t_-),\qquad & x<0,~ t>0,\\
(v_+,u_+,\t_+),\qquad & x>0,~ t>0.
\end{array}
\right.\label{(2.3)}
\end{equation}

As in \cite{Huang-Matsumura-Xin}, the viscous version of the above
contact discontinuity, called viscous contact wave
$(V^{CD},U^{CD},\T^{CD})(t,x)$, can be defined as follows. Since we
expect that
$$
P^{CD}\approx p_+=p_-, \quad \rm{and}\quad |U^{CD}|\ll1,
$$
the leading order of the energy equation $(1.1)_3$ is
$$
\f{R}{\g-1}\T_t+p_+U_x=\k(\f{\T_x}{V})_x.
$$
Thus, we can get the following nonlinear diffusion equation
$$
\T_t=a\v (\f{\T_x}{\T})_x,\quad \T(t,\pm)=\t_\pm,\quad a=\f{\nu
p_+(\g-1)}{R^2\g},
$$
which has a unique self-similar solution
$\hat\T(t,x)=\hat\T(\eta),~\eta=\f{x}{\sqrt{\v(1+t)}}$.

Now the viscous contact wave $(V^{CD},U^{CD},\T^{CD})(t,x)$ can be
defined by
\begin{equation}
\begin{array}{ll}
\di V^{CD}(t,x)=\f{R\hat\T(t,x)}{p_+},\\
\di U^{CD}(t,x)=u_+
+\f{\k(\g-1)}{R\g}\f{\hat\T_{x}(t,x)}{\hat\T(t,x)},\\
\di \T^{CD}(t,x)=\hat\T(t,x)+\f{\v[R\g-\nu(\g-1)]}{\g p_+}\hat\T_t.
\end{array}
\label{(2.4)}
\end{equation}
Here, it is straightforward to check that the viscous contact wave
defined in (\ref{(2.4)}) satisfies
\begin{equation}
|\T^{CD}-\t_\pm|+[\v(1+t)]^{\f12}|\T^{CD}_x|+\v(1+t)|\T^{CD}_{xx}|
=O(1)\d^{CD} e^{-\f{C_0x^2}{\v(1+t)}}, \label{(2.5)}
\end{equation}
as $|x|\ra+\i$, where $\d^{CD}=|\t_+-\t_-|$ represents the strength
of the viscous contact wave and $C_0$ is a positive generic
constant. Note that in the above definition, the higher order term
$\f{\v[R\g-\nu(\g-1)]}{\g p_+}\hat\T_t$ is used in
$\Theta^{CD}(t,x)$  so that the viscous contact wave
$(V^{CD},U^{CD},\T^{CD})(t,x)$ satisfies the momentum equation
exactly. Precisely,  $(V^{CD},U^{CD},\T^{CD})(t,x)$ satisfies the
system
\begin{equation}
\left\{
\begin{array}{l}
\di V^{\scriptscriptstyle CD}_{t}-U^{CD}_{x}=0,\\
\di U^{CD}_{t}+P^{CD}_{x}=\v(\f{U^{CD}_{x}}{V^{CD}})_x,\\
\di
\f{R}{\g-1}\T^{CD}_{t}+P^{CD}U^{CD}_{x}=\k(\f{\T^{CD}_{x}}{V^{CD}})_x+\v\f{(U^{CD}_{x})^2}{V^{CD}}+Q^{CD},
\end{array}
\right. \label{(2.6)}
\end{equation}
where $\di P^{CD}=\f{R\T^{CD}}{V^{CD}}$ and the error term $Q^{CD}$
has the property that
\begin{equation}
\begin{array}{ll}
\di  Q^{CD}&\di =O(1)\d^{CD}\v
(1+t)^{-2}e^{-\f{C_0x^2}{\v(1+t)}},\qquad {\rm as}~~|x|\ra +\i.
\end{array}
\label{(2.7)}
\end{equation}

\begin{remark} The viscous contact wave $(V^{CD},U^{CD},\T^{CD})(t,x)$
defined in \eqref{(2.4)} is different from the one used in
\cite{Huang-Matsumura-Xin} and \cite{Huang-Xin-Yang}. Here, this
ansatz is chosen such that the mass equation and the momentum
equation are satisfied exactly while the error term  occurs only in
the energy equation. However, note that the approximate energy
equation that the viscous contact wave satisfies is not in the
conservative form.
\end{remark}

\subsubsection{Rarefaction waves}

We now turn to the rarefaction waves. Since there is no exact
rarefaction wave profile for either the Navier-Stokes equations or
the Boltzmann equation, the following approximate rarefaction wave
profile satisfying the Euler equations was motivated by \cite{MN-86}
and \cite{Xin1}. For the completeness of the presentation, we
include its definition and the properties in this subsection.

If $(v_-, u_-, \theta_-) \in R_i (v_+, u_+, \theta_+) (i=1,3)$, then
there exists  a $i$-rarefaction wave $(v^{r_i}, u^{r_i},
\t^{r_i})(x/t)$ which is a global  solution of the following Riemann
problem
\begin{eqnarray}
\left\{
\begin{array}{l}
\di  v_{t}- u_{x}= 0,\\
\di u_{t} +  p_{x}(v, \theta) = 0 ,
\\
\di \frac{R}{\g-1}\theta_t + p(v, \theta)u_x =0,\\
\di  (v, u, \t)(t=0,x)=\left\{
\begin{array}{l}
\di (v_-, u_-, \t_-),   x<
0 ,\\
\di  (v_+, u_+, \t_+), x> 0 .
\end{array}
\right.
\end{array} \right.\label{(2.8)}
\end{eqnarray}
Consider the following inviscid Burgers equation with Riemann data
\begin{equation}
\left\{
\begin{array}{l}
w_t+ww_x=0,\\
w(t=0,x)=\left\{
\begin{array}{ll}
w_-,&x<0,\\
w_+,&x>0.
\end{array}
 \right.
  \end{array}
 \right.\label{(2.9)}
\end{equation}
If $w_-<w_+$, then the above Riemann problem admits a rarefaction
wave solution
\begin{equation}
w^r(t,x)=w^r(\f xt)=\left\{
\begin{array}{ll}
w_-,&\f xt\leq w_-,\\
\f xt,&w_-\leq \f xt\leq w_+,\\
w_+,&\f xt\geq w_+.
\end{array}
\right.\label{(2.10)}
\end{equation}
Obviously, we have the following Lemma,

\begin{lemma}\label{Lemma 2.1} For any shift $t_0>0$ in the time variable, we have
$$
|w^r(t+t_0,x)-w^r(t,x)|\leq \f{C}{t}t_0,
$$
where $C$ is a positive constant  depending only on $w_{\pm}$.
\end{lemma}

Remark that Lemma \ref{Lemma 2.1} plays an important role in the
wave interaction estimates for the rarefaction waves.

As in \cite{Xin1},  the approximate rarefaction wave $(V^R, U^R,
\Theta^R)(t,x)$ to the problem (\ref{(1.1)}) can be constructed by
the solution of the Burgers equation
\begin{eqnarray}
\left\{
\begin{array}{l}
\di w_{t}+ww_{x}=0,\\
\di w( 0,x
)=w_\s(x)=w(\f{x}{\s})=\f{w_++w_-}{2}+\f{w_+-w_-}{2}\tanh\f{x}{\s},
\end{array}
\right.\label{(2.11)}
\end{eqnarray}
where $\s>0$ is a small parameter
 to be determined. Note that the solution $w^r_\s(t,x)$ of the
problem (\ref{(2.11)}) is given by
$$
w^r_\s(t,x)=w_\s(x_0(t,x)),\qquad x=x_0(t,x)+w_\s(x_0(t,x))t.
$$
And $w^r_\s(t,x)$ has the following properties:

\begin{lemma}\label{Lemma 2.2} (\cite{Xin1}) Let $w_-<w_+$,
 $(\ref{(2.11)})$ has a unique smooth solution $w^r_\s(t,x)$
satisfying
\begin{enumerate}
\item[(1)] $w_-< w^r_\s(t,x)<w_+,~(w^r_\s)_x(t,x)\geq 0 $;

\item[(2)] For any $p$ $(1\leq p\leq +\infty)$, there exists a constant
$C$ such that
$$
\begin{array}{ll}
\| \f{\partial}{\partial x}w^r_\s(t,\cdot)\|_{L^p(\mathbf{R})}\leq
C\min\big{\{}(w_+-w_-)\s^{-1+ 1/p},~
(w_+-w_-)^{1/p}t^{-1+1/p}\big{\}}, \\[2mm]
 \| \f{\partial^2}{\partial x^2}w^r_\s(t,\cdot)\|_{L^p(\mathbf{R})}\leq
C\min\big{\{}(w_+-w_-)\s^{-2+ 1/p},~ \s^{-1+1/p}t^{-1}\big{\}};
\end{array}
$$
\item[(3)] If $ x-w_-t<0$ and $w_->0$, then
$$
\begin{array}{l}
 |w^r_\s(t,x)-w_-|\leq (w_+-w_-)e^{-\f{2|x-w_-t|}{\s}},\\[2mm]
 |\f{\partial}{\partial x}w^r_\s(t,\cdot)|\leq
\f{2(w_+-w_-)}{\s}e^{-\f{2|x-w_-t|}{\s}};
\end{array}
$$
 If $ x-w_+t> 0$ and $w_+<0$, then
$$
\begin{array}{l}
 |w^r_\s(t,x)-w_+|\leq (w_+-w_-)e^{-\f{2|x-w_+t|}{\s}},\\[2mm]
 |\f{\partial}{\partial x}w^r_\s(t,\cdot)|\leq
\f{2(w_+-w_-)}{\s}e^{-\f{2|x-w_+t|}{\s}};
\end{array}
$$

\item[(4)]$\sup\limits_{x\in\mathbf{R}} |w^r_\s(t,x)-w^r(\f xt)|\leq
\min\big\{(w_+-w_-),\f{\s}{t}[\ln(1+t)+|\ln\s|]\big\}$.
\end{enumerate}
\end{lemma}

Then  the smooth approximate rarefaction wave profile denoted by\\
$(V^{R_i}, U^{R_i}, \T^{R_i})(t,x)~(i=1,3)$ can be defined  by
\begin{eqnarray}
\left\{
\begin{array}{l}
\di  S^{R_i}(t,x)=s(V^{R_i}(t,x),\T^{R_i}(t,x))=s_+,\\
\di w_\pm=\l_{i\pm}:=\l_i(v_\pm,\t_\pm), \\
\di w_\s^r(t+t_0,x)= \l_i(V^{R_i}(t,x),s_+),\\
\di U^{R_i}(t,x)=u_+-\int^{V^{R_i}(t,x)}_{v_+}  \l_i(v,s_+) dv,
\end{array} \right.\label{(2.12)}
\end{eqnarray}
where $t_0$ is the shift used to control the interaction between
waves in different families with the property that  $t_0\ra 0$ as
$\v\ra0$. In the following,  we choose
\begin{equation}
t_0= \v^{\f15}, \qquad\mbox{\rm and}
\qquad\s=\v^{\f25}.\label{(2.13)}
\end{equation}

Note that $(V^{R_i}, U^{R_i}, \Theta^{R_i})(t,x)$ defined above
satisfies
\begin{eqnarray}
\begin{cases}
  V^{R_i}_t-U^{R_i} _{x} = 0,     \cr
    U^{R_i}_t+P^{R_i}_x
    =0,\cr
      \f{R}{\g-1}  \T^{R_i}_t
        + P^{R_i} U^{R_i}_x
    =0,
\end{cases}\label{(2.14)}
\end{eqnarray}
where $ P^{R_i}=p( V^{R_i}, \T  ^{R_i})$.

By Lemmas \ref{Lemma 2.1} and \ref{Lemma 2.2}, the properties on the
rarefaction waves can be summarized as follows.

\begin{lemma}\label{Lemma 2.3} The approximate rarefaction waves $(V^{R_i},
U^{R_i}, \Theta^{R_i})(t,x)~(i=1,3)$ constructed in \eqref{(2.12)}
have the following properties:
\begin{enumerate}
\item[(1)] $U^{R_i}_x(t,x)>0$ for $x\in \mathbf{R}$, $t>0$;
\item[(2)] For any $1\leq p\leq +\i,$ the following estimates holds,
$$
\begin{array}{ll}
\|(V^{R_i},U^{R_i}, \Theta^{R_i})_x\|_{L^p(dx)} \leq
C(t+t_0)^{-1+\f1p},\\
\|(V^{R_i},U^{R_i}, \Theta^{R_i})_{xx}\|_{L^p(dx)} \leq
C\s^{-1+\f1p}(t+t_0)^{-1},\\
\|(V^{R_i},U^{R_i}, \Theta^{R_i})_{xxx}\|_{L^p(dx)} \leq
C\s^{-2+\f1p}(t+t_0)^{-1},\\
\end{array}
$$
where the positive constant $C$ only depends on $p$ and the wave
strength;
\item[(3)] If $x\geq \l_{1+}(t+t_0)$, then
$$
\begin{array}{l}
 |(V^{R_1},U^{R_1},\T^{R_1})(t,x)-(v_-,u_-,\t_-)|\leq Ce^{-\f{2|x-\l_{1+}(t+t_0)|}{\s}},\\[2mm]
 |(V^{R_1},U^{R_1},\T^{R_1})_x(t,x)|\leq
\f{C}{\s}e^{-\f{2|x-\l_{1+}(t+t_0)|}{\s}};
\end{array}
$$
If $x\leq \l_{3-}(t+t_0)$, then
$$
\begin{array}{l}
 |(V^{R_3},U^{R_3},\T^{R_3})(t,x)-(v_+,u_+,\t_+)|\leq Ce^{-\f{2|x-\l_{3-}(t+t_0)|}{\s}},\\[2mm]
 |(V^{R_3},U^{R_3},\T^{R_3})_x(t,x)|\leq
\f{C}{\s}e^{-\f{2|x-\l_{3-}(t+t_0)|}{\s}};
\end{array}
$$

\item[(4)] There exist positive constants $C$ and $\s_0$ such that
for $\s\in(0,\s_0)$ and $t,t_0>0,$
$$
\sup_{x\in\mathbf{R}}|(V^{R_i},U^{R_i},
\Theta^{R_i})(t,x)-(v^{r_i},u^{r_i}, \theta^{r_i})(\f xt)|\leq
\f{C}{t}[\s\ln(1+t+t_0)+\s|\ln\s|+t_0].
$$
\end{enumerate}
\end{lemma}

\subsubsection{Superposition of rarefaction waves and contact
discontinuity} In this subsection, we will define the solution
profile that consists of the superposition of two rarefaction waves
and a contact discontinuity. Let $ (v_-,u_-,\t_-) \in$
$R_1$-$CD$-$R_3(v_+,u_+,\t_+)$. Then there exist uniquely two
intermediate states $(v_*, u_*,\t_*)$ and $(v^*, u^*,\t^*)$ such
that $(v_-, u_-,\t_-)\in R_1(v_*, u_*,\t_*)$, $(v_*, u_*,\t_*)\in
CD(v^*, u^*,\t^*)$  and $(v^*, u^*,\t^*)\in R_3(v_+,u_+,\t_+)$.

So the  wave pattern $(\bar V,\bar U,\bar\T)(t,x)$ consisting of
1-rarefaction wave, 2-contact discontinuity and 3-rarefaction wave
that solves the corresponding
 Riemann problem of the Euler system \eqref{(1.2+)} can be defined
by
\begin{eqnarray}
 \left(\begin{array}{cc} \bar V\\ \bar U \\
 \bar  \T
\end{array}
\right)(t, x)= \left(\begin{array}{cc}v^{r_1}+ v^{cd}+ v^{r_3}\\ u^{r_1}+ u^{cd}+ u^{r_3} \\
\t^{r_1}+ \t^{cd}+ \t^{r_3}
\end{array}
\right)(t, x) -\left(\begin{array}{cc} v_*+v^*\\ u_*+u^*\\
\t_*+\t^*
\end{array}
\right) ,\label{(2.15)}
\end{eqnarray}
where $(v^{r_1}, u^{r_1}, \t^{r_1} )(t,x)$ is the 1-rarefaction wave
defined in \eqref{(2.8)} with the right state $(v_+, u_+, \t_+)$
replaced by $(v_*, u_*, \theta_* )$, $(v^{cd}, u^{cd}, \t^{cd}
)(t,x)$ is the contact discontinuity   defined in \eqref{(2.3)} with
the states $(v_-, u_-, \t_-)$ and $(v_+, u_+, \t_+)$ replaced by
$(v_*, u_*, \theta_* )$ and $(v^*, u^*, \theta^* )$ respectively,
and $(v^{r_3}, u^{r_3}, \t^{r_3})(t,x)$ is the 3-rarefaction wave
defined in \eqref{(2.8)} with the left state $(v_-, u_-, \t_-)$
replaced by $(v^*, u^*, \theta^* )$.

Correspondingly, the approximate  wave pattern $(V,U,\T)(t,x)$ of
the compressible Navier-Stokes equations can be defined by
\begin{eqnarray}
 \left(\begin{array}{cc} V\\ U \\
  \T
\end{array}
\right)(t, x)= \left(\begin{array}{cc} V^{R_1}+ V^{CD}+ V^{R_3}\\ U^{R_1}+ U^{CD}+ U^{R_3} \\
\T^{R_1}+ \T^{CD}+ \T^{R_3}
\end{array}
\right)(t, x) -\left(\begin{array}{cc} v_*+v^*\\ u_*+u^*\\
\t_*+\t^*
\end{array}
\right) ,\label{(2.16)}
\end{eqnarray}
where $(V^{R_1}, U^{R_1}, \Theta^{R_1} )(t,x)$ is the approximate
1-rarefaction wave defined in \eqref{(2.12)} with the right state
$(v_+, u_+, \t_+)$ replaced by $(v_*, u_*, \theta_* )$, $(V^{CD},
U^{CD}, \Theta^{CD} )(t,x)$ is the viscous contact wave   defined in
\eqref{(2.4)} with the states $(v_-, u_-, \t_-)$ and $(v_+, u_+,
\t_+)$ replaced by $(v_*, u_*, \theta_* )$ and $(v^*, u^*, \theta^*
)$ respectively, and $(V^{R_3}, U^{R_3}, \Theta^{R_3} )(t,x)$ is the
approximate 3-rarefaction wave defined in \eqref{(2.12)} with the
left state $(v_-, u_-, \t_-)$ replaced by $(v^*, u^*, \theta^* )$.

Thus, from the construction of the contact wave and Lemma \ref{Lemma
2.3}, we have the following relation between the approximate wave
pattern $(V,U,\T)(t,x)$ of the compressible Navier-Stokes equations
and the exact inviscid  wave pattern
 $(\bar V,\bar U,\bar\T)(t,x)$ to the
Euler equations
\begin{equation}
\begin{array}{ll}
\di |(V,U,\T)(t,x)-(\bar V,\bar U,\bar\T)(t,x)|\\
\di \quad\leq
\f{C}{t}[\s\ln(1+t+t_0)+\s|\ln\s|+t_0]+C\d^{CD}e^{-\f{cx^2}{\v(1+t)}},
\end{array}\label{(2.17)}
\end{equation}
with $t_0=\v^{\f15}$ and $\s=\v^{\f25}$.

Moreover,  $(V,U,\T)(t,x)$ satisfies the following system
\begin{eqnarray}
\begin{cases}
  V_t-U _{x} = 0,     \cr
    U_t+P_x
    = \v (\frac{U_{x}}{V}) _{x}+Q_1,\cr
        \f{R}{\g-1} \T _t+P U_x
    =\k( \frac{\T_{x}}{ V})_x + \v   \frac{  U_x^2}{ V}
     +Q_2,
\end{cases}\label{(2.18)}
\end{eqnarray}
 where $ P =p( V , \T  )$, and
$$\begin{array}{ll}
   \di Q_1&\di=(P-P^{R_1}-P^{CD}-P^{R_3})_x-\v(\f{U_x}{V}-\f{U^{CD}_x}{V^{CD}})_x,\\
  \di Q_2&\di= (PU_x-P^{R_1}U^{R_1}_x-P^{CD}U^{CD}_x-P^{R_3}U^{R_3}_x)-\k(\f{\T_x}{V}-\f{\T^{CD}_x}{V^{CD}})_x\\
  &\di-\v(\f{U_x^2}{ V}-\f{(U^{CD}_x)^2}{ V^{CD}})-Q^{CD}.
\end{array}$$

Direct calculation shows that
\begin{equation}
\begin{array}{lll}
\di Q_1&=&\di O(1)\Big[|(V^{R_1}_x,\T^{R_1}_x)||(V^{CD}-v_*,\T^{CD}-\t_*,V^{R_3}-v^*,\T^{R_3}-\t^*)|\\
&&\di
+|(V^{R_3}_x,\T^{R_3}_x)||(V^{R_1}-v_*,\T^{R_1}-\t_*,V^{CD}-v^*,\T^{CD}-\t^*)|\\
&&\di+|(V^{CD}_x,\T^{CD}_x,U^{CD}_{xx})||(V^{R_1}-v_*,\T^{R_1}-\t_*,V^{R_3}-v^*,\T^{R_3}-\t^*)|\\
&&\di
+\v|(U^{CD}_x,V^{CD}_x)||(U^{R_1}_x,V^{R_1}_x,U^{R_3}_x,V^{R_3}_x)|
+\v|(U^{R_1}_x,V^{R_1}_x)||(U^{R_3}_x,V^{R_3}_x)|\Big]\\
&&\di+O(1)\v\Big[|U^{R_1}_{xx}|+|U^{R_3}_{xx}|+|U^{R_1}_x||V^{R_1}_x|+|U^{R_3}_x||V^{R_3}_x|\Big]\\
& :=&\di Q_{11}+Q_{12}.
\end{array}
\label{(2.19)}
\end{equation}
Similarly, we have
\begin{equation}
\begin{array}{lll}
\di Q_2&=&\di O(1)\Big[|U^{R_1}_x||(V^{CD}-v_*,\T^{CD}-\t_*,V^{R_3}-v^*,\T^{R_3}-\t^*)|\\
&&\di
+|U^{R_3}_x||(V^{R_1}-v_*,\T^{R_1}-\t_*,V^{CD}-v^*,\T^{CD}-\t^*)|\\
&&\di+|(U^{CD}_{x},V^{CD}_x,\T^{CD}_x)||(V^{R_1}-v_*,\T^{R_1}-\t_*,V^{R_3}-v^*,\T^{R_3}-\t^*)|\\
&&\di
+\v|(U^{CD}_x,V^{CD}_x,\T^{CD}_x)||(U^{R_1}_x,V^{R_1}_x,\T^{R_1}_x,U^{R_3}_x,V^{R_3}_x,\T^{R_1}_x)|\\
&&\di
+\v|(U^{R_1}_x,V^{R_1}_x,\T^{R_1}_x)||(U^{R_3}_x,V^{R_3}_x,\T^{R_3}_x)|\Big]\\
&&\di +O(1)\v\Big[|\T^{R_1}_{xx}|+|\T^{R_3}_{xx}|+|(U^{R_1}_x,V^{R_1}_x,\T^{R_1}_x,U^{R_3}_x,V^{R_3}_x,\T^{R_3}_x)|^2\Big]+|Q^{CD}|\\
& :=&\di Q_{21}+Q_{22}+|Q^{CD}|.
\end{array}
\label{(2.20)}
\end{equation}
Here $Q_{11}$ and $Q_{21}$ represent the interactions  coming from
different wave patterns, $Q_{12}$ and $Q_{22}$ represent the error
terms coming from the approximate  rarefaction wave profiles, and
$Q^{CD}$ is the error term defined in \eqref{(2.7)} due to the
viscous contact wave.

Firstly, we estimate the interaction terms $Q_{11}$ and $Q_{21}$ by
dividing the whole domain
$\Omega=\{(t,x)|(t,x)\in\mathbf{R}^+\times\mathbf{R}\}$ into three
regions:
$$
\begin{array}{l}
\Omega_-=\{(t,x)|2x\leq \l_{1*}(t+t_0)\},\\
\Omega_{CD}=\{(t,x)|\l_{1*}(t+t_0)<2x<\l_3^*(t+t_0)\},\\
\Omega_+=\{(t,x)|2x\geq \l_3^*(t+t_0)\},
\end{array}
$$
where $\l_{1*}=\l_1(v_*,\t_*)$ and $\l_3^*=\l_3(v^*,\t^*)$.

 Now from Lemma \ref{Lemma 2.3}, we have the following estimates in each
section:
\begin{itemize}
\item In $\Omega_-$,

$\begin{array}{ll}|V^{R_3}-v^*|&=O(1)e^{-\f{2|x|+2\l_3^*(t+t_0)}{\s}}\\
&=O(1)e^{-\l_3^*\v^{-1/5}}e^{-\f{2|x|+\l_3^*(t+t_0)}{\v^{2/5}}},
\end{array}$

$\begin{array}{ll}|(V^{CD}-v_*,V^{CD}-v^*)|&=O(1)\d^{CD}e^{-\f{C[\l_{1*}(t+t_0)]^2}{4\v(1+t)}}\\
&=O(1)e^{-\f{Ct_0(t+t_0)}{\v}}\\
&=O(1)e^{-\f{Ct_0(|x|+t+t_0)}{\v}}\\
&=O(1)e^{-C\v^{-3/5}}e^{-\f{C(|x|+t+t_0)}{\v^{4/5}}};
\end{array}
$
\item In $\Omega_{CD}$,

$\begin{array}{ll}|V^{R_1}-v_*|&=O(1)e^{-\f{2|x|+2|\l_{1*}|(t+t_0)}{\s}}\\
&=O(1)e^{-|\l_{1*}|\v^{-1/5}}e^{-\f{2|x|+|\l_{1*}|(t+t_0)}{\v^{2/5}}},
\end{array}$

$\begin{array}{ll}|V^{R_3}-v^*|&=O(1)e^{-\f{2|x|+2\l_3^*(t+t_0)}{\s}}\\
&=O(1)e^{-\l_3^*\v^{-1/5}}e^{-\f{2|x|+\l_3^*(t+t_0)}{\v^{2/5}}};
\end{array}$
\item In $\Omega_+$,

$\begin{array}{ll}|V^{R_1}-v_*|&=O(1)e^{-\f{2|x|+2|\l_{1*}|(t+t_0)}{\s}}\\
&=O(1)e^{-|\l_{1*}|\v^{-1/5}}e^{-\f{2|x|+2|\l_{1*}|(t+t_0)}{\v^{2/5}}},
\end{array}$

$\begin{array}{ll}|(V^{CD}-v_*,V^{CD}-v^*)|&=O(1)\d^{CD}e^{-\f{C[\l_3^*(t+t_0)]^2}{4\v(1+t)}}\\
&=O(1)e^{-\f{Ct_0(t+t_0)}{\v}}\\
&=O(1)e^{-\f{Ct_0(|x|+t+t_0)}{\v}}\\
&=O(1)e^{-C\v^{-3/5}}e^{-\f{C(|x|+t+t_0)}{\v^{4/5}}}.
\end{array}
$
\end{itemize}
Hence,  in summary, we have
\begin{equation}
|(Q_{11},Q_{21})|=O(1)e^{-C\v^{-1/5}}e^{-\f{C(|x|+t+t_0)}{\v^{2/5}}},\label{(2.21)}
\end{equation}
for some positive constants $C$.

Now we consider the system \eqref{(1.1)} with the initial values
\begin{equation}
(v,u,\t)(t=0,x)=(V,U,\T)(t=0,x). \label{(2.22)}
\end{equation}
Introduce the following scaled variables
\begin{equation}
y=\f{x}{\v},\quad \tau=\f{t}{\v}. \label{(2.23)}
\end{equation}
In the following, we will use the notations $(v,u,\t)(\tau,y)$ and
$(V,U,\T)(\tau,y)$ for the unknown functions and the approximate
wave profiles in the scaled variables. Set the perturbation around
the composite wave pattern $(V,U,\T)(\tau,y)$ by
$$
(\phi,\psi,\zeta)(\tau,y)=(v-V,u-U,\t-\T)(\tau,y).
$$
Then the perturbation $(\phi,\psi,\zeta)(\tau,y)$ satisfies the
system
\begin{equation}
\left\{
\begin{array}{ll}
\di \phi_{\tau}-\psi_{y}=0, \\
\di \psi_{\tau}+(p-P)_{y}=(\frac {u_{y}}v-\f{U_y}{V})_{y}-\v Q_1,\\
\di \f{R}{\g-1}\z_{\tau}+
(pu_y-PU_y)=\nu(\frac{\theta_{y}}{v}-\f{\T_y}{V})_y+(\f{u_y^2}{v}-\frac{U^2_{y}}{V})-\v
Q_2,\\
\di (\phi,\psi,\z)(\tau=0,y)=0.
\end{array}
\right.
 \label{(2.24)}
\end{equation}
And this system will be studied in Section 3.

\subsubsection{Main result to the compressible
Navier-Stokes equations}

We are now ready to  state the main result on the compressible
Navier-Stokes equations as follows.

\begin{theorem}\label{Theorem 2.1} Given a Riemann solution $(\bar V,\bar
U,\bar\T)(t,x)$ defined in \eqref{(2.15)}, which is a superposition
of two rarefaction waves and a contact discontinuity for the Euler
system (\ref{(1.2+)}), there exist small positive constants
$\delta_0$ and $\v_0$ such that if the contact wave strength
$\delta^{CD}\leq \delta_0$ and the viscosity coefficient $\v\leq
\v_0$,
 then the compressible Navier-Stokes equations (\ref{(1.1)}) with \eqref{(1.2)} and \eqref{(1.3)}
  admits a unique global solution
$(v^\v,u^\v,\t^\v)(t,x)$ satisfying
\begin{equation}
\sup_{(t,x)\in\Sigma_h}|(v^\v,u^\v,\t^\v)(t,x)-(\bar V,\bar
U,\bar\T)(t,x)|\leq C_h~\v^{\f15},\quad \forall h>0, \label{(2.25)}
\end{equation}
where $\Sigma_h=\{(t,x)|t\geq h, \f{x}{\sqrt{1+t}}\geq h
\v^{\a},0<\a<\f12\}$, and the positive constant $C_h$ depends only
on
 $h$ but is
independent of $\v$.
\end{theorem}

\begin{remark} Theorem \ref{Theorem 2.1} shows that, away from the initial time
$t=0$ and the contact discontinuity located at $x=0$ with the
expansion rate $\f{x^2}{\v(1+t)}$, for the viscosity coefficient
$\v<\v_0$, there exists a unique global solution
$(v^\v,u^\v,\t^\v)(t,x)$  of the compressible Navier-Stokes
equations (\ref{(1.1)}) which tends to the Riemann solution $(\bar
V,\bar U,\bar\T)(t,x)$ consisting of two rarefaction waves and a
contact discontinuity when $\v\rightarrow 0$ and
$\k=O(\v)\rightarrow 0$. Moreover, a uniform convergence rate
$\v^{\f{1}{5}}$ holds on the set $\Sigma_h$ for any $h>0$.
\end{remark}

\begin{remark} Theorem \ref{Theorem 2.1} holds uniformly when $(t,x)\in
 \Sigma_h$ for any fixed $h>0$ if the contact wave strength $\d^{CD}$
and the viscosity coefficient $\v$ are suitably small. However,  if
we restrict the problem to a set $\Sigma_h \cap \{t\leq T\}$ for any
fixed $T>0$, then we do not need to impose the smallness condition
on the contact wave strength $\d^{CD}$ because one can apply the
Gronwall inequality to get an estimate depending on time $T$ rather
than the uniform estimate in time.
\end{remark}

\subsection{Boltzmann equation}

We now turn to the Boltzmann equation. Similarly, we also define
individual wave pattern, and then the superposition and finally
state the main result in this subsection.

\subsubsection{Contact discontinuity}
We first recall the construction of the contact wave
$(V^{CD},U^{CD},\T^{CD})(t,x)$ for the Boltzmann equation in
\cite{Huang-Xin-Yang}. Consider the Euler system \eqref{(1.16)} with
a Riemann initial data
\begin{equation}
(v,u,\t)(t=0,x)=\left\{
\begin{array}{l}
(v_-,u_{-},\t_-),~~~x<0,\\
(v_+,u_{+},\t_+),~~~x>0,
\end{array}
\right. \label{(2.26)}
\end{equation}
where $u_{\pm}=(u_{1\pm},0,0)$ and $v_\pm>0,\t_\pm>0,u_{1\pm}$ are
given constants. It is known (cf. \cite{Smoller}) that the Riemann
problem (\ref{(1.16)}), (\ref{(2.26)}) admits a contact
discontinuity solution
\begin{equation}
(v^{cd},u^{cd},\t^{cd})(t,x)=\left\{
\begin{array}{l}
(v_-,u_{-},\t_-),~~~x<0,\\
(v_+,u_{+},\t_+),~~~x>0,
\end{array}
\right. \label{(2.27)}
\end{equation}
provided that
\begin{equation}
u_{1+}=u_{1-},\qquad p_-:=\f{2\t_-}{3v_-}=p_+:=\f{2\t_+}{3v_+}.
\label{(2.28)}
\end{equation}
Motivated by (\ref{(2.27)}) and (\ref{(2.28)}), we expect that for
the contact wave\\ $(V^{CD},U^{CD},\T^{CD})(t,x)$,
$$
P^{CD}=\f{2\T^{CD}}{3V^{CD}}\approx p_+,~~~|U^{CD}|^2\ll1.
$$
Then the leading order of the energy equation $(\ref{(1.22)})_4$ is
\begin{equation}
\t_t+p_+u_{1x}=\v(\f{\lambda(\t)\t_x}{v})_x. \label{(2.29)}
\end{equation}
By using the mass equation $(\ref{(1.22)})_1$ and
$v\approx\f{R\t}{p_+}$, we obtain the following nonlinear diffusion
equation
\begin{equation}
\t_t=\v(a(\t)\t_x)_x,~~~a(\t)=\f{9 p_+\lambda(\t)}{10\t}.
\label{(2.30)}
\end{equation}
From \cite{Atkinson-Peletier} and \cite{Duyn-Peletier}, we know that
the nonlinear diffusion equation (\ref{(2.30)}) admits a unique
self-similar solution $\hat{\T}(\eta),~\eta=\f{x}{\sqrt{\v(1+t)}}$
with the following boundary conditions
$$
\hat{\T}(-\i,t)=\t_-,~~\hat{\T}(+\i,t)=\t_+.
$$
Let $\delta=|\t_+-\t_-|$. $\hat{\T}(t,x)$ has the property
\begin{equation}
\hat{\T}_x(t,x)=\f{O(1)\delta^{CD}}{\sqrt{\v(1+t)}}e^{-\f{cx^2}{\v(1+t)}},~~~~~~{\rm
as}~~~x\rightarrow\pm\i, \label{(2.31)}
\end{equation}
with some positive constant $c$  depending only on  $\t_{\pm}$.

Now  the contact wave $(V^{CD},U^{CD},\T^{CD})(t,x)$ can be defined
by
\begin{equation}
\begin{array}{ll}
\di V^{CD}=\f{2}{3p_+}\hat{\T},\\
\di U^{CD}_1=u_{1+}+\f{2\v
a(\hat{\T})}{3p_+}\hat{\T}_x,~~~~U^{CD}_i=0,(i=2,3),\\[3mm]
~~~\T^{CD}=\hat{\T}+\f{2\v}{3p_+}\hat{\T}_t[\f43\mu(\hat\T)-\f35\l(\hat\T)].
\end{array}
 \label{(2.32)}
\end{equation}
Note that the contact wave $(V^{CD},U^{CD},\T^{CD})(t,x)$ satisfies
the following system
\begin{equation}
\left\{\begin{array}{llll}
\di V^{CD}_t-U^{CD}_{1x}=0,\\
\di U^{CD}_{1t}+P^{CD}_x=\f{4\v}{3}(\f{\mu(\T^{CD})}{V^{CD}}U^{CD}_{1x})_x+Q^{CD}_1,\\
\di U^{CD}_{it}=\v(\f{\mu(\T^{CD})}{V^{CD}}U^{CD}_{ix})_x,
i=2,3,\\
\di\T^{CD}_{t}+
P^{CD}U^{CD}_{1x}=\v(\f{\lambda(\T^{CD})}{V^{CD}}\T^{CD}_x)_x+\f{4\v}{3}\f{\mu(\T^{CD})}{V^{CD}}
(U^{CD}_{1x})^2\\
\quad\di+\v\sum_{i=2}^3\f{\mu(\T^{CD})}{V^{CD}}(U^{CD}_{ix})^2+Q^{CD}_2,
\end{array}
\right. \label{(2.33)}
\end{equation}
where
\begin{equation}
Q^{CD}_1=\f{4\v}{3}(\f{\mu(\T^{CD})-\mu(\hat\T)}{V^{CD}}U^{CD}_{1x})_x=\di
O(1)\delta^{CD}\v^{\f32}(1+t)^{-\f52}e^{-\f{cx^2}{\v(1+t)}},\label{(2.34)}
\end{equation}
\begin{equation}
\begin{array}{ll}
 Q^{CD}_2&\di=[\f{2\v}{3p_+}\hat{\T}_t(\f43\mu(\hat\T)-\f35\l(\hat\T))]_t
+\f{2\v}{3p_+V^{CD}}\hat{\T}_t[\f43\mu(\hat\T)-\f35\l(\hat\T)]U^{CD}_{1x}\\
&\di\quad +
\f{\v}{V^{CD}}(\lambda(\hat\T)\hat\T_x-\lambda(\T^{CD})\T^{CD}_x)_x-\f{4\v\mu(\T^{CD})}{3V^{CD}}
(U^{CD}_{1x})^2\\
&\di=O(1)\delta^{CD}\v(1+t)^{-2}e^{-\f{cx^2}{\v(1+t)}},
\end{array}
\label{(2.35)}
\end{equation}
with some positive constant $c>0$  depending only on $\t_\pm$.

\begin{remark} The viscous contact wave
$(V^{CD},U^{CD},\T^{CD})(t,x)$ for the Boltzmann equation
\eqref{(1.4)} defined in \eqref{(2.32)} is different from the one
used in \cite{Huang-Xin-Yang}. Here, this ansatz is chosen such that
the momentum equation is satisfied with a higher order error term.
This is also different from the compressible Navier-Stokes equations
where the ansatz satisfies the momentum equation exactly. But
similar to the compressible Navier-Stokes cases, the approximate
energy equation that the viscous contact wave satisfies is not in
the conservative form.
\end{remark}

From (\ref{(2.31)}), we have
\begin{equation}
\left\{
\begin{array}{l}
|\hat\T-\t_-|= O(1)\delta^{CD} e^{-\f{cx^2}{2\v(1+t)}},~~~~~{\rm if}~x<0,\\
|\hat\T-\t_+|= O(1)\delta^{CD} e^{-\f{cx^2}{2\v(1+t)}},~~~~~{\rm
if}~x>0.
\end{array}
\right. \label{(2.36)}
\end{equation}
Therefore,
\begin{equation}
|(V^{CD},U^{CD},\T^{CD})(t,x)-(v^{cd},u^{cd},\t^{cd})(t,x)|= O(1)\delta^{CD} e^{-\f{cx^2}{2\v(1+t)}}.\\
\label{(2.37)}
\end{equation}

\subsubsection{Rarefaction waves}
The construction of the $i$-rarefaction wave\\
$(V^{R_i},U^{R_i},\T^{R_i})(t,x)~(i=1,3)$ to the Boltzmann equation
is almost same as the one defined in  \eqref{(2.14)} for the
compressible Navier-Stokes equations in the previous section. By
setting $U^{R_i}_j=0$ for $i=1,3$ and $j=2,3$, all the properties of
the approximate rarefaction waves
$(V^{R_i},U_1^{R_i},\T^{R_i})(t,x)~(i=1,3)$  given in Lemma 2.3 will
also be used later.

\subsubsection{Superposition of rarefaction waves and contact discontinuity}

We now consider the superposition of two rarefaction waves and a
contact discontinuity. Set $(v_-,u_-,\t_-) \in$
$R_1$-$CD$-$R_3(v_+,u_+,\t_+)$. Then there exist uniquely two
intermediate states $(v_*, u_*,\t_*)$ and $(v^*, u^*,\t^*)$ such
that $(v_*, u_*,\t_*)\in R_1(v_-, u_-,\t_-)$, $(v_*, u_*,\t_*)\in
CD(v^*, u^*,\t^*)$  and $(v^*, u^*,\t^*)\in R_3(v_+,u_+,\t_+)$.

So the wave pattern $(\bar V,\bar U,\bar\T)(t,x)$ consisting of
1-rarefaction wave, 2-contact discontinuity and 3-rarefaction wave
as a  Riemann solution to  the Euler system \eqref{(1.16)} can be
defined by
\begin{equation}
\begin{array}{l}
 \left(\begin{array}{cc} \bar V\\ \bar U_1 \\
 \bar  \T
\end{array}
\right)(t, x)= \left(\begin{array}{cc}v^{r_1}+ v^{cd}+ v^{r_3}\\ u_1^{r_1}+ u_1^{cd}+ u_1^{r_3} \\
\t^{r_1}+ \t^{cd}+ \t^{r_3}
\end{array}
\right)(t, x) -\left(\begin{array}{cc} v_*+v^*\\ u_{1*}+u_1^*\\
\t_*+\t^*
\end{array}
\right),
\\[7mm]
\di \bar U_i=0,(i=2,3).
\end{array}
\label{(2.38)}
\end{equation}
where $(v^{r_1}, u_1^{r_1}, \t^{r_1} )(t,x)$ is the approximate
1-rarefaction wave defined in \eqref{(2.8)} with the right state
$(v_+, u_+, \t_+)$ replaced by $(v_*, u_{1*}, \theta_* )$, $(v^{cd},
u_1^{cd}, \t^{cd} )(t,x)$ is the contact discontinuity defined in
\eqref{(2.27)} with the states $(v_-, u_-, \t_-)$ and $(v_+, u_+,
\t_+)$ replaced by $(v_*, u_{*}, \theta_* )$ and $(v^*, u^*,
\theta^* )$ respectively, and $(v^{r_3}, u_1^{r_3}, \t^{r_3})(t,x)$
is the 3-rarefaction wave defined in \eqref{(2.8)} with the left
state $(v_-, u_-, \t_-)$ replaced by $(v^*, u_1^*, \theta^* )$.

Correspondingly, the approximate superposition wave $(V,U,\T)(t,x)$
can be defined by
\begin{equation}
\begin{array}{l}
 \left(\begin{array}{cc} V\\ U_1 \\
  \T
\end{array}
\right)(t, x)= \left(\begin{array}{cc}  V^{R_1}+ V^{CD}+ V^{R_3}\\ U_1^{R_1}+ U_1^{CD}+ U_1^{R_3} \\
\T^{R_1}+ \T^{CD}+ \T^{R_3}
\end{array}
\right)(t, x) -\left(\begin{array}{cc} v_*+v^*\\ u_{1*}+u_1^{*}\\
\t_*+\t^*
\end{array}
\right), \\[7mm]
\di U_i=0,(i=2,3).
\end{array}
\label{(2.39)}
\end{equation}
where $(V^{R_1}, U_1^{R_1}, \Theta^{R_1} )(t,x)$ is the
1-rarefaction wave defined in \eqref{(2.12)} with the right state
$(v_+, u_+, \t_+)$ replaced by $(v_*, u_{1*}, \theta_* )$, $(V^{CD},
U_1^{CD}, \Theta^{CD} )(t,x)$ is the viscous contact wave   defined
in \eqref{(2.32)} with the states $(v_-, u_-, \t_-)$ and $(v_+, u_+,
\t_+)$ replaced by $(v_*, u_*, \theta_* )$ and $(v^*, u^*, \theta^*
)$ respectively, and $(V^{R_3}, U_1^{R_3}, \Theta^{R_3} )(t,x)$ is
the approximate 3-rarefaction wave defined in \eqref{(2.12)} with
the left state $(v_-, u_-, \t_-)$ replaced by $(v^*, u_1^*, \theta^*
)$.

Thus, from the construction of the contact wave and Lemma \ref{Lemma
2.3}, we have the following relation between the approximate wave
pattern $(V,U,\T)(t,x)$ of the Boltzmann equation and the exact
inviscid wave pattern $(\bar V,\bar U,\bar\T)(t,x)$ to the Euler
equations
\begin{equation}
\begin{array}{ll}
\di |(V,U,\T)(t,x)-(\bar V,\bar U,\bar\T)(t,x)|\\
\di \quad\leq
\f{C}{t}[\s\ln(1+t+t_0)+\s|\ln\s|+t_0]+C\d^{CD}e^{-\f{cx^2}{\v(1+t)}},
\end{array}\label{(2.40)}
\end{equation}
with $t_0=\v^{\f15}$ and $\s=\v^{\f25}$.

Then we have
\begin{equation}
\left\{
\begin{array}{l}
  V_t-U _{1x} = 0,  \\
    U_{1t}+P_x
    = \v (\frac{\mu(\T)U_{1x}}{V}) _{x}+Q_1,\\
U_{it}=\v(\f{\mu(\T)U_{ix}}{V})_x,~~i=2,3,
    \\
  \T _t+P U_{1x}
    =\v( \frac{\l(\T)\T_{x}}{ V})_x + \v   \frac{ \mu(\T) U_{1x}^2}{ V}
     +Q_2,
\end{array}
\right.\label{(2.41)}
\end{equation}
 where $ P =p( V , \T  )$ and
$$\begin{array}{ll}
   \di Q_1&\di=(P-P^{R_1}-P^{CD}-P^{R_3})_x-\v(\f{\mu(\T)U_{1x}}{V}-\f{\mu(\T^{CD})U^{CD}_{1x}}{V^{CD}})_x-Q^{CD}_1,\\
  \di Q_2&\di= (PU_{1x}-P^{R_1}U^{R_1}_{1x}-P^{CD}U^{CD}_{1x}-P^{R_3}U^{R_3}_{1x})-\v(\f{\l(\T)\T_x}{V}-\f{\l(\T^{CD})\T^{CD}_x}{V^{CD}})_x\\
  &\di\qquad-\v(\f{\mu(\T)U_{1x}^2}{V}-\f{\mu(\T^{CD})(U^{CD}_{1x})^2}{ V^{CD}})-Q_2^{CD}.
\end{array}$$
Direct computation yields
\begin{equation}
\begin{array}{lll}
   \di Q_1&=&\di O(1)\Big[|(V^{R_1}_x,\T^{R_1}_x)||(V^{CD}-v_*,\T^{CD}-\t_*,V^{R_3}-v^*,\T^{R_3}-\t^*)|\\
&&\di
+|(V^{R_3}_x,\T^{R_3}_x)||(V^{R_1}-v_*,\T^{R_1}-\t_*,V^{CD}-v^*,\T^{CD}-\t^*)|\\
&&\di+|(V^{CD}_x,\T^{CD}_x,U^{CD}_{xx})||(V^{R_1}-v_*,\T^{R_1}-\t_*,V^{R_3}-v^*,\T^{R_3}-\t^*)|\\
&&\di
+\v|(U^{CD}_x,V^{CD}_x)||(U^{R_1}_x,V^{R_1}_x,U^{R_3}_x,V^{R_3}_x)|
+\v|(U^{R_1}_x,V^{R_1}_x)||(U^{R_3}_x,V^{R_3}_x)|\Big]\\
&&\di+O(1)\v\Big[|U^{R_1}_{xx}|+|U^{R_3}_{xx}|+|U^{R_1}_x||V^{R_1}_x|+|U^{R_3}_x||V^{R_3}_x|\Big]+|Q^{CD}_1|\\
& :=&\di Q_{11}+Q_{12}+|Q^{CD}_1|,
\end{array}
\label{(2.42)}
\end{equation}
and
\begin{equation}\begin{array}{lll}
   \di Q_2&=&\di O(1)\Big[|U^{R_1}_x||(V^{CD}-v_*,\T^{CD}-\t_*,V^{R_3}-v^*,\T^{R_3}-\t^*)|\\
&&\di
+|U^{R_3}_x||(V^{R_1}-v_*,\T^{R_1}-\t_*,V^{CD}-v^*,\T^{CD}-\t^*)|\\
&&\di+|(U^{CD}_{x},V^{CD}_x,\T^{CD}_x)||(V^{R_1}-v_*,\T^{R_1}-\t_*,V^{R_3}-v^*,\T^{R_3}-\t^*)|\\
&&\di
+\v|(U^{CD}_x,V^{CD}_x,\T^{CD}_x)||(U^{R_1}_x,V^{R_1}_x,\T^{R_1}_x,U^{R_3}_x,V^{R_3}_x,\T^{R_1}_x)|\\
&&\di
+\v|(U^{R_1}_x,V^{R_1}_x,\T^{R_1}_x)||(U^{R_3}_x,V^{R_3}_x,\T^{R_3}_x)|\Big]\\
&&\di +O(1)\v\Big[|\T^{R_1}_{xx}|+|\T^{R_3}_{xx}|+|(U^{R_1}_x,V^{R_1}_x,\T^{R_1}_x,U^{R_3}_x,V^{R_3}_x,\T^{R_3}_x)|^2\Big]+|Q_2^{CD}|\\
& :=&\di Q_{21}+Q_{22}+|Q_2^{CD}|.
\end{array}\label{(2.43)}
\end{equation}
Here, $Q_{11}$ and $Q_{21}$ represent the interaction of waves in
different families, $Q_{12}$ and $Q_{22}$ represent the error terms
coming from the approximate  rarefaction wave profiles, and
$Q_i^{CD}(i=1,2)$ are the error terms defined in \eqref{(2.34)} and
\eqref{(2.35)} due to the viscous contact wave.

Similar to the compressible Navier-Stokes equations case, for the
interaction terms, we have
\begin{equation}
|(Q_{11},Q_{21})|=O(1)e^{-C\v^{-1/5}}e^{-\f{C(|x|+t+t_0)}{\v^{2/5}}},\label{(2.44)}
\end{equation}
for some positive constants $C$.

We now reformulate the system by introducing  a scaling for the
independent variables. Set
$$
y=\f x\v,~~\tau=\f t\v
$$
as in the previous section for the compressible Navier-Stokes
equations. We also use the notations $(v,u,\t)(\tau,y), \mb
G(\tau,y,\x), \Pi_1(\tau,y,\x)$ and $(V,U,\T)(\tau,y)$ in the scaled
independent variables. Set the perturbation around the composite
wave $(V,U,\T)(\tau,y)$ by
$$
(\phi,\psi,\zeta)(\tau,y)=(v-V,u-U,\t-\T)(\tau,y).
$$
Under this scaling, the hydrodynamic limit problem is reduced to a
time asymptotic stability problem of the composite wave to the
Boltzmann equation. Notice that the hydrodynamic limit proved here
is global in time compared to  the case on shock profile studied in
\cite{Yu} which is locally in time.
%However, we do not know whether there exists
%some  appropriate scaling for the shock profile so that this method
%can be applied.

%With the above scaling, the proof of Theorem 2.1 will be given by
%energy method as \cite{Huang-Li-Matsumura} for  the scaled
%perturbation $(\p,\psi,\z)(\tau,y)$ and $\mb{G}(\tau,y,\xi)$.

From (\ref{(1.22)}) and (\ref{(2.42)}), we have the following system
for the perturbation $(\phi,\psi,\z)$
\begin{equation}
\left\{
\begin{array}{ll}
\di \phi_{\tau}-\psi_{1y}=0, \\
\di \psi_{1\tau}+(p-P)_{y}=\f43(\frac {\mu(\t)u_{1y}}v-\f{\mu(\T)U_{1y}}{V})_{y}-\int\xi_1^2\Pi_{1y}d\xi-\v Q_1,\\
\di \psi_{i\tau}=(\frac {\mu(\t)u_{i1y}}v-\f{\mu(\T)U_{iy}}{V})_{y}-\int\xi_1\xi_i\Pi_{1y}d\xi,~~i=2,3,\\
\di \z_{\tau}+
(pu_{1y}-PU_{1y})=(\frac{\l(\t)\theta_{y}}{v}-\f{\l(\T)\T_y}{V})_y+\f43(\frac
{\mu(\t)u_{1y}^2}v-\f{\mu(\T)U_{1y}^2}{V})\\
\di\qquad+\sum_{i=2}^3\frac
{\mu(\t)u_{iy}^2}v+\sum_{i=1}^3u_i\int\xi_1\xi_i\Pi
_{1y}d\xi-\int\xi_1\f{|\x|^2}{2}\Pi_{1y}d\xi-\v U_1 Q_1-\v
Q_2,\\
\end{array}
\right.
 \label{(2.45)}
\end{equation}
where the error terms $Q_i~(i=1,2)$ are given in (\ref{(2.42)}) and
(\ref{(2.43)}) respectively.

We now derive the equation for the  non-fluid component $\mb{G}
(\tau,y,\xi)$ in the scaled independent variables. From
(\ref{(1.18)}), we have
\begin{equation}
\di \mb{G} _{\tau}-\f{u_1}{v}\mb{G}
_y+\f{1}{v}\mb{P}_1(\xi_1\mb{M}_y) +\f{1}{v}\mb{P}_1(\xi_1\mb{G}
_y)=\mb{L}_\mb{M}\mb{G} +Q(\mb{G} ,\mb{G} ).\label{(2.46)}
\end{equation}
Thus, we obtain
\begin{equation}
\mb{G} =\f{1}{v}\mb{L}^{-1}_\mb{M}[ \mb{P}_1(\xi_1
\mb{M}_y)]+\Pi_1,\label{(2.47)}
\end{equation}
and
\begin{equation}
\Pi_1(\tau,y,\xi)=\mb{L}_\mb{M}^{-1}[\mb{G} _{\tau}-\f{u_1}v\mb{G}
_y+\f{1}{v}\mb{P}_1(\xi_1\mb{G} _y) -Q(\mb{G} ,\mb{G}
)].\label{(2.48)}
\end{equation}
Let
\begin{equation}
\mb{G}
_0(\tau,y,\xi)=\f{3}{2v\t}\mb{L}^{-1}_\mb{M}\{\mb{P}_1[\xi_1(\f{|\xi-u|^2}{2\t}{\T}_y+\xi\cdot{U}_{y})\mb{M}]\},
\label{(2.49)}
\end{equation}
and
\begin{equation}
\mb{G} _1(\tau,y,\xi)=\mb{G} (\tau,y,\xi)-\mb{G} _0(\tau,y,\xi).
\label{(2.50)}
\end{equation}
Then $\mb{G} _1(\tau,y,\xi)$ satisfies
\begin{equation}
\begin{array}{ll}
\mb{G} _{1\tau}-\mb{L}_\mb{M}\mb{G}
_1=&\di-\f{3}{2v\t}\mb{P}_1[\xi_1
(\f{|\xi-u|^2}{2\t}\z_y+\xi\cdot\psi_y)\mb{M}]\\
&\di+\f{u_1}{v}\mb{G} _y-\f{1}{v}\mb{P}_1(\xi_1\mb{G} _y)+Q(\mb{G}
,\mb{G} )-\mb{G} _{0\tau}.
\end{array}
\label{(2.51)}
\end{equation}
Notice that in \eqref{(2.50)} and \eqref{(2.51)}, $\mb{G} _0$ is
subtracted from $\mb{G} $ because $\|(\T_y,U_y)\|^2\sim
(1+\v^{\f12}\tau)^{-1/2}$ is not integrable globally in $\tau$.

Finally, from (\ref{(1.15)}) and the scaling transformation
\eqref{(2.23)}, we have
\begin{equation}
\di f_{\tau}-\f{u_1}{v}f_y+\f{\xi_1}{v}f_y=Q(f,f). \label{(2.52)}
\end{equation}

The estimation on the fluid and non-fluid components governed by the
above systems will be given in the last section.

\subsubsection{Main result to Boltzmann equation}

With  the above preparation, we are now ready to  state the main
result on the Boltzmann equation as follows.

\begin{theorem}\label{Theorem 2.2} Given a Riemann solution $(\bar V,\bar
U,\bar\T)(t,x)$ defined in \eqref{(2.38)}, which is a superposition
of two rarefaction waves and a contact discontinuity to the Euler
system (\ref{(1.16)}), there exist small positive constants
$\delta_0$, $\v_0$ and a global Maxwellian
$\mb{M}_\star=\mb{M}_{[v_\star,u_\star,\t_\star]}$, such that if the
contact wave strength $\d^{CD}\leq \delta_0$, and the Knudsen number
$\v\leq \v_0$,
 then the
Boltzmann equation (\ref{(1.4)}) admits a unique global solution
$f^\v(t,x,\xi)$ satisfying
\begin{equation}
\sup_{(t,x)\in\Sigma_h}\|f^\v(t,x,\xi)-\mb{M}_{[\bar V,\bar
U,\bar\T]}(t,x,\xi)\|_{L_\xi^2(\f{1}{\sqrt{\mb{M}_\star}})}\leq C_h~
\v^\f15,\qquad \forall h>0,\label{(2.53)}
\end{equation}
where $\Sigma_h=\{(t,x)|t\geq h, \f{x}{\sqrt{1+t}}\geq h
\v^{\a},0<\a<\f12\}$, the norm
$\|\cdot\|_{L_\xi^2(\f{1}{\sqrt{\mb{M}_\star}})}$ is
 $\|\f{\cdot}{\sqrt{\mb{M}_\star}}\|_{L_\xi^2(\mb{R}^3)}$ and the positive constant $C_h$ depends only on  $h$ but
is independent of $\v$.
\end{theorem}

\begin{remark} Theorem \ref{Theorem 2.2} shows that, away from the initial time
$t=0$ and the contact discontinuity located at $x=0$ with the
expansion rate $\f{x^2}{\v(1+t)}$, for Knudsen number $\v<\v_0$,
there exists a unique global solution $f^\v(t,x,\xi)$ of the
Boltzmann equation (\ref{(1.4)}) which tends to the Maxwellian
$\mb{M}_{[\bar V,\bar U,\bar\T]}(t,x,\xi)$ with $(\bar V,\bar
U,\bar\T)(t,x)$ being the Riemann solution to the Euler equation
with the combination of two rarefaction waves and a contact
discontinuity when $\v\rightarrow 0$. Moreover, a uniform
convergence rate $\v^{\f{1}{5}}$ in the norm
$L_\xi^2(\f{1}{\sqrt{\mb{M}_\star}})$ holds on the set $\Sigma_h$
for any fixed $h>0$.
\end{remark}

\begin{remark} Theorem \ref{Theorem 2.2} holds uniformly on the $(t,x)\in
 \Sigma_h$ for any $h>0$ if the contact wave strength $\d^{CD}$ and Knudsen number $\v$ are suitably small.
But if we restrict the problem to  the set $\Sigma_h \cap \{t\leq
T\}$ for any fixed $T>0$, then we don't need the smallness condition
on the contact wave strength $\d^{CD}$ by using Gronwall inequality
to get a time dependent estimate rather than  the uniform estimation
in time.
\end{remark}

\noindent\textbf{Notations:}  Throughout this paper, the positive
generic constants which are independent of $T,\v$ are denoted by
$c$, $C$ or $C_0$. For function spaces, $H^l(\mathbf{R})$ denotes
the $l$-th order Sobolev space with its norm
$$
\|f\|_l=(\sum^l_{j=0}\|\partial^j_yf\|^2)^\frac{1}{2}, \quad {\rm
and}~\|\cdot\|:=\|\cdot\|_{L^2(dy)},
$$
where $L^2(dz)$ means the $L^2$ integral over $\mathbf{R}$ with
respect to the Lebesgue measure $dz$, and $z=x$ or $y$.

\section{Proof of Theorem \ref{Theorem 2.1}: Zero dissipation limit of Navier-Stokes equations}
\setcounter{equation}{0}

We will prove Theorem \ref{Theorem 2.1} about
 the fluid dynamic limit for the compressible Navier-Stokes
equations to the Riemann solution of the Euler equations
 in this section. The proof is based on
the energy estimates on the perturbation in the scaled independent
variables. In fact, to prove Theorem \ref{Theorem 2.1}, it is
sufficient to prove the following theorem.

\begin{theorem}\label{Theorem 3.1} There exist small positive constants
$\delta_1$ and $\v_1$ such that if the initial values and the
contact wave strength $\delta^{CD}$ satisfy
\begin{equation}
\mathcal{N}(\tau)|_{\tau=0}+\delta^{CD} \le \delta_1, \label{(4.9)}
\end{equation}
and the Knudsen number $\v$ satisfies $\v\leq \v_1$, then the
problem (\ref{(2.24)}) admits a unique global solution
$(v^\v,u^\v,\t^\v)(\tau,y)$ satisfying
\begin{equation}
\begin{array}{l}
\di
\sup_{\tau,y}|(v^\v,u^\v,\t^\v)(\tau,y)-(V,U,\Theta)(\tau,y)|\leq
C\v^{\f15}.\\
\end{array}
\label{(4.10)}
\end{equation}
Here $\mathcal{N}(\tau)$ is defined by \eqref{(3.3+)} below.
\end{theorem}

We will focus on the reformulated system (\ref{(2.24)}). Since the
local existence of the solution to (\ref{(2.24)}) is standard, to
prove the global existence, we only need to close the following a
priori estimate by the continuity argument
\begin{equation}
\begin{array}{ll}
\di \mathcal{N}(\tau)=&\di\sup_{0\leq \tau^\prime\leq \tau}
\|(\p,\psi,\z)(\tau^\prime,\cdot)\|_1^2\leq \chi^2,
\end{array}
\label{(3.3+)}
\end{equation}
where $\chi$ is a small positive constant depending only on the
initial values and the strength of the contact wave. And the proof
of the above a priori estimate is given by the following energy
estimations.

Firstly, multiplying $\eqref{(2.24)}_2$ by $\psi$ yields
\begin{equation}
\begin{array}{ll}
{\displaystyle(\frac12\psi^2)_\tau-(p-P)\psi_{y}+(\frac{u_{y}}{v}
-\frac{U_y}{V})\psi_{y}=-\v
Q_1\psi+\left[(\frac{u_y}{v}-\frac{U_y}{V})\psi-(p-P)\psi\right]_y.}\end{array}\label{(3.1)}
\end{equation}
Since $p-P=R\T(\frac1v-\frac1{V})+\frac{R\zeta}{v}$ and
$\phi_\tau=\psi_y$, we get
\begin{equation}
\begin{array}{l}
\di(\frac12\psi^2)_\tau-R\T(\frac{1}{v}-\frac{1}{V})\phi_\tau-\frac{R}{v}\zeta\psi_{y}
+\frac{\psi^2_y}{v}\\
\di=-(\frac{1}{v}-\frac{1}{V}) U_{y}\psi_{y}-\v
Q_1\psi+\left[(\frac{u_y}{v}-\frac{U_y}{V})\psi-(p-P)\psi\right]_y.
\end{array}
\label{(3.2)}
\end{equation}
Set
\begin{equation}
\Phi(z)=z-1-\ln z.\label{(3.3)}
\end{equation}
It is easy to check that $\Phi(1)=\Phi'(1)=0$ and $\Phi(z)$ is
strictly convex around $z=1$. Moreover,
\begin{equation}
\di [R\T\Phi(\frac{v}{V})]_\tau=R\T_\tau\Phi(\frac{v}{V})
-R\T(\frac{1}{v}-\frac{1}{V})\phi_\tau-\f{PV_\tau}{vV}\phi^2.
\label{(3.4)}
\end{equation}

On the other hand, note that
\begin{equation}
[\f{R}{\g-1}\T
\Phi(\frac{\theta}{\T})]_\tau=\f{R}{\g-1}(1-\frac{\T}{\theta})\zeta_\tau+
\f{R}{\g-1}\Phi(\frac{\theta}{\T})\T_\tau-\f{R}{\g-1}\f{\T_\tau\z^2}{\t\T},\label{(3.5)}
\end{equation}
and
\begin{equation}
\begin{array}{ll}
&\di \frac{R}{\gamma-1}(1-\frac{\T}{\theta})\zeta_\tau\\ &\di
=(1-\frac{\T}{\theta})[-(pu_{y}-PU_{y})+
\nu(\frac{\theta_y}{v}-\frac{\T_y}{V})_y+(\frac{u_y^2}{v}-\frac{U_y^2}{V})-\v Q_2]\\
&\di= -\frac{R}{v}\zeta\psi_{y}-\frac{\zeta}{\theta}(p-P)U_{y}
-\nu(\frac{\zeta}{\theta})_y(\frac{\theta_y}{v}-\frac{\T_y}{V})
+\frac{\zeta}{\theta}(\frac{u_y^2}{v}-\frac{U_y^2}{V})\\
&\di \quad-\v\frac{\zeta}{\theta}Q_2+\left[\nu\frac{\zeta}{\theta}(\frac{\theta_y}{v}-\frac{\T_y}{V})\right]_y\\
&\di= -\frac{R}{v}\zeta\psi_{y}-\frac{\zeta}{\theta}(p-P)U_{y}
-\frac{\nu\zeta_y^2}{v\theta}-\nu\f{\z_y}{\t}(\frac1v-\f1V)\T_y\\
&\di
\quad+\frac{\nu\zeta\theta_y}{\theta^2}(\frac{\theta_y}{v}-\frac{\T_y}{V})+\frac{\zeta}{\theta}(\frac{u_y^2}{v}-\frac{U_y^2}{V})-\v
Q_2\frac{\zeta}{\theta}+\left[\frac{\nu\zeta}{\theta}(\frac{\theta_y}{v}-\frac{\T_y}{V})\right]_y.
\end{array}\label{(3.6)}
\end{equation}
Substituting \eqref{(3.4)}-\eqref{(3.6)} into \eqref{(3.2)} gives
\begin{equation}
\begin{array}{l}
\di [\frac12\psi^2+R\T\Phi(\frac{v}{V})+\frac{R}{\gamma-1}\T
\Phi(\frac{\theta}{\T})]_\tau +\frac{\psi_{y}^2}{v} \di
+\frac{\nu\zeta_y^2}{v\theta}+J_1\\
\di =-U_y(\f1v-\f1V)\psi_y
-\nu\f{\z_y}{\t}(\frac1v-\f1V)\T_y+\frac{\nu\zeta\theta_y}{\theta^2}(\frac{\theta_y}{v}-\frac{\T_y}{V})\\
\di\quad+\frac{\zeta}{\theta}(\frac{u_y^2}{v}-\frac{U_y^2}{V})-\v
Q_1\psi-\v Q_2\frac{\zeta}{\theta}+(\cdots)_y,
\end{array}\label{(3.7)}
\end{equation}
where
\begin{equation}
J_1=\frac{\zeta}{\theta}(p-P)U_{y}-R\T_\tau\Phi(\f
vV)-\f{R}{\g-1}\T_\tau\Phi(\f\t\T)+\f{PV_\tau}{vV}\phi^2+\f{R}{\g-1}\f{\T_\tau\z^2}{\t\T}.\label{(3.8)}
\end{equation}
Direct calculation shows that
\begin{equation}
\begin{array}{ll}
J_1&\di =PU_y[\Phi(\f{\t V}{v\T})+\g\Phi(\f
vV)]-[\f{U_y^2}{V}+\nu(\f{\T_y}{V})_y+\v Q_2][(\g-1)\Phi(\f
vV)-\Phi(\f\T\t)]\\
&\di=PU_y[\Phi(\f{\t V}{v\T})+\g\Phi(\f
vV)]-[\f{U_y^2}{V}+\nu(\f{\T_y}{V})_y+\v Q_2][(\g-1)\Phi(\f
vV)-\Phi(\f\T\t)].
\end{array}\label{(3.10)}
\end{equation}
Thus, substituting \eqref{(3.10)} into \eqref{(3.7)} gives
\begin{equation}
\begin{array}{l}
\di [\frac12\psi^2+R\T\Phi(\frac{v}{V})+\frac{R}{\gamma-1}\T
\Phi(\frac{\theta}{\T})]_\tau +\frac{\psi_{y}^2}{v} \di
+\frac{\nu\zeta_y^2}{v\theta}\\
\di+P(U^{R_1}_y+U^{R_3}_y)[\Phi(\f{\t V}{v\T})+\g\Phi(\f vV)]=J_2-\v
Q_1\psi-\v Q_2\frac{\zeta}{\theta}+(\cdots)_y,
\end{array}\label{(3.11)}
\end{equation}
where
\begin{equation}
\begin{array}{ll}
\di J_2=\di -PU^{CD}_y[\Phi(\f{\t V}{v\T})+\g\Phi(\f
vV)]+[\f{U_y^2}{V}+\nu(\f{\T_y}{V})_y+\v Q_2][(\g-1)\Phi(\f
vV)-\Phi(\f\T\t)]\\
\qquad\di -U_y(\f1v-\f1V)\psi_y -\nu\f{\z_y}{\t}(\frac1v-\f1V)\T_y
+\frac{\nu\zeta\theta_y}{\theta^2}(\frac{\theta_y}{v}-\frac{\T_y}{V})+\frac{\zeta}{\theta}(\frac{u_y^2}{v}-\frac{U_y^2}{V}).
\end{array}\label{(3.12)}
\end{equation}
Here,  $(\cdots)_y$ represents the conservative terms which vanishes
after integrating in $y$ over $\mathbf{R}$.

By the strict convexity of  $\Phi(z)$ around $z=1$, under
 the a priori assumption \eqref{(3.3+)} with sufficiently small $\chi>0$,
there exist positive constants $c_1$ and $c_2$ such that,
\begin{equation}
\begin{array}{l}
\di c_1\phi^2\le \Phi(\frac{v}{V})\le c_2\phi^2,\quad c_1\zeta^2\le
\Phi(\frac{\Theta}{\theta}),\Phi(\frac{\theta}{\Theta})\le
c_2\zeta^2,\\
\di c_1(\phi^2+\z^2)\leq \Phi(\f{\t V}{v\T})\leq c_2(\phi^2+\z^2).
\end{array}\label{(3.13)}
\end{equation}
Thus, we have
\begin{equation}
\begin{array}{l}
\di \int_{\mb{R}}|J_2|dy\leq
\int_{\mb{R}}(\f{\psi_y^2}{4v}+\f{\nu\z_y^2}{4v\t})dy+C(\tau+\tau_0)^{-2}\|(\phi,\z)\|^2\\
\di +C\int_{\mb{R}}\d^{CD}\v(1+\v\tau)^{-1}e^{-\f{c\v
y^2}{1+\v\tau}}|(\p,\z)|^2dy+\int_{\mb{R}}\v|Q_2||(\p,\z)|^2dy.
\end{array}\label{(3.14)}
\end{equation}
Notice that the last term $\v|Q_2||(\p,\z)|^2$ on the right hand
side of \eqref{(3.14)} can be estimated similarly as for
 the terms
$\v Q_1\psi$ and $\v Q_2\f{\z}{\t}$ under the a priori assumption
\eqref{(3.3+)}. Now we estimate the terms $\v Q_1\psi$ and $\v
Q_2\f{\z}{\t}$ on the right hand side of \eqref{(3.11)}. First,
$$
\int_{\mb{R}}\v |Q_1||\psi|dy=\int_{\mb{R}}\v
(|Q_{11}|+|Q_{12}|)|\psi|dy.
$$
From the estimation on the interaction given in \eqref{(2.21)}, we
get
\begin{equation}
\begin{array}{ll}
&\di \int_{0}^\tau\int_{\mb{R}}\v |Q_{11}||\psi|d\tau dy\\
&\di \leq
\int_{0}^\tau\|\psi\|_{L^\i_y}\int_{\mathbf{R}}|Q_{11}|dxd\tau\\
&\di \leq
C\int_{0}^\tau e^{-C\v^{-1/5}}e^{-\f{C(t+t_0)}{\v^{2/5}}}\|\psi\|^{\f12}\|\psi_y\|^{\f12}d\tau\\
&\di \leq \b\int_{0}^\tau\|\psi_y\|^2
d\tau+C_\b\int_{0}^\tau e^{-C\v^{-1/5}}e^{-C\v^{3/5}(\tau+\tau_0)}\|\psi\|^{\f23}d\tau\\
&\di \leq \b\int_{0}^\tau\|\psi_y\|^2 d\tau+C_\b
e^{-C\v^{-1/5}}\sup_{[0,\tau]}\|\psi(\tau)\|^\f23\\
&\di \leq \b\int_{0}^\tau\|\psi_y\|^2
d\tau+\b\sup_{[0,\tau]}\|\psi(\tau)\|^2+C_\b e^{-C\v^{-1/5}},
\end{array}
\label{(3.15)}
\end{equation}
 and
\begin{equation}
\begin{array}{ll}
\di &\di \int_{0}^\tau\int_{\mb{R}}\v |Q_{12}||\psi|d\tau dy\\
&\di \leq
\v^2\int_{0}^\tau\int_{\mb{R}}(|(w_\d^r)_{xx}|,|(w_\d^r)_x|^2)|\psi|d\tau dy\\
&\di \leq
\v\int_{0}^\tau(\|(w_\d^r)_{xx}\|_{L^1(dx)},\|(w_\d^r)_x\|_{L^2(dx)}^2)\|\psi\|_{L^\i_y}d\tau\\
&\di \leq
\int_{0}^\tau(\tau+\tau_0)^{-1}\|\psi\|^{\f12}\|\psi_y\|^{\f12}d\tau\\
&\di \leq \b\int_{0}^\tau\|\psi_y\|^2
d\tau+C_\b\int_{0}^\tau (\tau+\tau_0)^{-\f43}\|\psi\|^{\f23}d\tau\\
&\di \leq \b\int_{0}^\tau\|\psi_y\|^2 d\tau+3C_\b
\tau_0^{-\f13}\sup_{[0,\tau]}\|\psi(\tau)\|^\f23\\
&\di \leq \b\int_{0}^\tau\|\psi_y\|^2
d\tau+\b\sup_{[0,\tau]}\|\psi(\tau)\|^2+C_\b \v^{\f25},
\end{array}
\label{(3.16)}
\end{equation}
where $\tau_0=\f{t_0}{\v}=\v^{-\f45}$, and $\b>0$ is a small
constant to be determined later and $C_\b$ is a positive constant
depending on $\b$.

The term $\v Q_2\f\z\t$ can be estimated similarly because the only
difference is about the error term $Q^{CD}$ coming from  the viscous
contact wave in $Q_2$. For this, we have
\begin{equation}
\begin{array}{ll}
\di \v\int_{0}^\tau\int_{\mb{R}}|Q^{CD}||\z|dyd\tau\\
\di \leq
\v^2\int_{0}^\tau\Big[\|\z\|_{L^\i_y}\int_{\mb{R}}(1+\v\tau)^{-2}e^{-\f{c\v
y^2}{1+\v\tau}}dy\Big]d\tau\\
\di \leq
\v^{\f32}\int_{0}^\tau\Big[\|\z\|_{L^2_y}^\f12\|\z_y\|^\f12_{L^2_y}(1+\v\tau)^{-\f32}\Big]d\tau\\
\di \leq
\b\int_0^\tau\|\z_y\|^2d\tau+C_\b\v^2\sup_{[0,\tau]}\|\z\|_{L^2_y}^\f23\int_{0}^\tau(1+\v\tau)^{-2}d\tau\\
\di \leq \b\|\z_y\|^2+\b\sup_{[0,\tau]}\|\z\|_{L^2_y}^2+C_\b
\v^{\f32}.
\end{array}
\label{(3.17)}
\end{equation}
By substituting \eqref{(3.13)}-\eqref{(3.17)} into \eqref{(3.11)}
and choosing $\b$ suitably small, we can get
\begin{equation}
\begin{array}{ll}
&\di \|(\phi,\psi,\zeta)(\tau,\cdot)\|^2+\int_0^\tau\Big[\|(\psi_y,\zeta_y)\|^2+\|\sqrt{(U^{R_1}_y,U^{R_3}_y)}(\phi,\zeta)\|^2\Big]d\tau\\
&\di \leq
C\int_0^\tau(\tau+\tau_0)^{-2}\|(\phi,\zeta)\|^2d\tau+C\v^{\f25}\\
&\di \quad
+C\d^{CD}\v\int_0^\tau\int_{\mathbf{R}}(1+\v\tau)^{-1}e^{-\f{C_0\v
y^2}{1+\v\tau}}|(\p,\z)|^2dyd\tau.
\end{array}
 \label{(3.18)}
\end{equation}

Now we need to estimate $\|\phi_y\|^2$. Let $\tilde{v}=\f vV$, then
$$
\f{\tilde{v}_\tau}{\tilde{v}}=\f{u_y}{v}-\f{U_y}{V}.
$$
Rewrite the equation $\eqref{(2.24)}_2$ as
\begin{equation}
(\f{\tilde{v}_y}{\tilde{v}})_\tau-\psi_\tau-(p-P)_y-\v
Q_1=0.\label{(3.19)}
\end{equation}
By multiplying \eqref{(3.19)} by $\f{\tilde{v}_y}{\tilde{v}}$ and
noticing that
\begin{equation}
-(p-P)_y=\f{R\t}{v}\f{\tilde{v}_y}{\tilde
v}-\f{R\z_y}{v}+(p-P)\f{V_y}{V}+R\T_y(\f1v-\f1V),\label{(3.20)}
\end{equation}
we  get
$$
\begin{array}{ll}
&\di \left[\f{1}{2}(\f{\tilde{v}_y}{\tilde
v})^2-\psi\f{\tilde{v}_y}{\tilde
v}\right]_\tau+\left[\psi\f{\tilde{v}_\tau}{\tilde
v}\right]_y+\f{R\t}{v}(\f{\tilde{v}_y}{\tilde v})^2\\[5mm]
=&\di
\psi_y(\f{u_y}{v}-\f{U_y}{V})+\left[\f{R\z_y}{v}-(p-P)\f{V_y}{V}-R\T_y(\f1v-\f1V)+\v
Q_1\right]\f{\tilde{v}_y}{\tilde v}.
\end{array}
$$

Integrating the above equality over $[0,\tau]\times\mathbf{R}$ in
 $\tau$ and $y$, we obtain
\begin{equation}
\begin{array}{ll}
&\di \int_{\bf R}\left[\f{1}{2}(\f{\tilde{v}_y}{\tilde
v})^2-\psi\f{\tilde{v}_y}{\tilde
v}\right](\tau,y)dy+\int_0^\tau\int_{\bf
R}\f{R\t}{2v}(\f{\tilde{v}_y}{\tilde v})^2dy d\tau\\[4mm]
\leq &\di
C\int_0^\tau\bigg[\|(\psi_y,\z_y)\|^2+\v^2\|Q_1\|^2\bigg]d\tau+C\int_0^\tau\int_{\bf
R}|(V_y,U_y,\T_y)|^2|(\p,\z)|^2dy d\tau.
\end{array}\label{(3.21)}
\end{equation}
The by using the equality
$$
\f{\tilde v_y}{\tilde
v}=\f{v_y}{v}-\f{V_y}{V}=\f{\p_y}{v}-\f{V_y\p}{vV},
$$
we have
\begin{equation}
C^{-1}(|\p_y|^2-|V_y\p|^2)\leq (\f{\tilde{v}_y}{\tilde v})^2\leq
C(|\p_y|^2+|V_y\p|^2).\label{(3.22)}
\end{equation}
By the estimation on $Q_{11}$ in \eqref{(2.21)} and Lemma 2.3, we
have
\begin{equation}
\begin{array}{ll}
\di \int_0^\tau\v^2\|Q_1\|^2d\tau&\di \leq
C\int_0^\tau\int_{\mathbf{R}}\v^2(|Q_{11}|^2+|Q_{12}|^2)dyd\tau\\
&\di \leq C\int_0^t\int_{\mathbf
R}(|Q_{11}|^2+\v^2|(w_\d^r)_{xx}|^2+\v^2|(w_\d^r)_x|^4)dxdt\\
&\di\leq Ce^{-C\v^{-1/5}}+C\v^2(t_0^{-2}+\d^{-1}t_0^{-1})\\
&\di\leq C\v^{\f75}.
\end{array}
\label{(3.23)}
\end{equation}
Moreover, we have
\begin{equation}
\begin{array}{ll}
\di |(V_y,U_y,\T_y)|^2= \v^2|(V_x,U_x,\T_x)|^2\\
\qquad\di \leq
\v^2\sum_{i=1,3}|(V^{R_i}_x,U^{R_i}_x,\T^{R_i}_x)|^2+\v^2|(V_x^{CD},U^{CD}_x,\T_x^{CD})|^2\\
\qquad \di \leq
C\v^2(t+t_0)^{-2}+C\d^{CD}\v(1+t)^{-1}e^{-\f{C_0x^2}{\v(1+t)}}\\
\qquad\di =C(\tau+\tau_0)^{-2}+C\d^{CD}\v(1+\v\tau)^{-1}e^{-\f{C_0\v
y^2}{1+\v\tau}}.
\end{array}\label{(3.24)}
\end{equation}

Substituting \eqref{(3.22)}-\eqref{(3.24)} into \eqref{(3.21)} gives
\begin{equation}
\begin{array}{ll}
&\di \|\phi_y(\tau,\cdot)\|^2+\int_0^\tau\|\p_y\|^2 d\tau \leq
C\|(\p,\psi)(\tau,\cdot)\|^2\\
&\di\quad+C\int_0^\tau\|(\psi_y,\z_y)\|^2d\tau
+C\int_0^\tau(\tau+\tau_0)^{-2}\|(\phi,\zeta)\|^2d\tau+C\v^{\f75}\\
&\di\quad +C\d^{CD}\int_0^\tau\int_{\bf
R}\v(1+\v\tau)^{-1}e^{-\f{C_0\v y^2}{1+\v\tau}}|(\p,\z)|^2dy d\tau.
\end{array}
 \label{(3.25)}
\end{equation}
Now we estimate the higher order derivatives of $(\psi,\zeta)$.
Multiplying $\eqref{(2.24)}_2$ by $-\psi_{yy}$ and
$\eqref{(2.24)}_3$ by $-\zeta_{yy}$, and then adding the resulting
equations together yield
\begin{equation}
\begin{array}{l}
\di
[\frac12\psi_y^2+\frac{R}{2(\gamma-1)}\zeta_y^2]_\tau+\frac{\psi_{yy}^2}{v}
+\nu\frac{\zeta_{yy}^2}{v}\\
\di=\big\{(p-P)_y+\frac{v_y}{v^2}
\psi_y+[U_y(\f1v-\f1V)]_y+\v Q_1\big\}\psi_{yy}\\
\di +\big\{(pu_y-PU_y)+\frac{\nu v_y}{v^2} \zeta_y
+[\nu\Theta_y(\f1v-\f1V)]_y+(\frac{u_y^2}{v}-\frac{U_y^2}{V})+\v
Q_2\big\}\zeta_{yy}.
\end{array}
\label{(3.26)}
\end{equation}
The right hand side of \eqref{(3.26)} will be estimated terms by
terms as follows. From \eqref{(3.20)} and \eqref{(3.24)}, we  get
\begin{equation}
\begin{array}{ll}
\di \int_0^\tau\int_{\mathbf
R}(p-P)_y\psi_{yy}dyd\tau\\
\di \leq C\int_0^\tau\int_{\mathbf R}
\Big[|(\phi_y,\zeta_y)|+|(V_y,\T_y)||(\phi,\zeta)|\Big]|\psi_{yy}|dyd\tau\\
\di\leq
\b\int_0^\tau\|\psi_{yy}\|^2d\tau+C_\b\int_0^\tau\|(\phi_y,\zeta_y)\|^2d\tau+C_\b\int_0^\tau(\tau+\tau_0)^{-2}\|(\phi,\zeta)\|^2d\tau\\
\di +C_\b\d^{CD}\int_0^\tau\int_{\bf R}\v(1+\v\tau)^{-1}e^{-\f{C\v
y^2}{1+\v\tau}}|(\p,\z)|^2dy d\tau.
\end{array}\label{(3.27)}
\end{equation}
Similar estimate holds for the term $\int_0^\tau\int_{\mathbf
R}(pu_y-PU_y)\zeta_{yy}dyd\tau.$

 Notice that
\begin{equation}
\begin{array}{ll}
\di \int_0^\tau\int_{\mathbf
R}\frac{v_y}{v^2}\psi_y\psi_{yy}dyd\tau\\
\di \leq C\int_0^\tau\int_{\mathbf R}
(|\phi_y|+|V_y|)|\psi_y||\psi_{yy}|dyd\tau\\
\di \leq
C\int_0^\tau(\|\phi_y\|\|\psi_{yy}\|\|\psi_y\|_{L_y^\i}+\|V_y\|_{L^\i_y}\|\|\psi_y\|\psi_{yy}\|)d\tau\\
\di \leq
C\int_0^\tau\|\psi_{yy}\|^{\f32}\|\psi_y\|^{\f12}\|\phi_y\|d\tau+C\v^{\f12}\int_0^\tau
\|\psi_y\|\psi_{yy}\| d\tau\\
\di\leq
\b\int_0^\tau\|\psi_{yy}\|^2d\tau+C_\b(\sup_{[0,\tau]}\|\phi_y\|^4+\v)\int_0^\tau\|\psi_y\|^2d\tau\\
\di\leq
\b\int_0^\tau\|\psi_{yy}\|^2d\tau+C_\b(\chi^4+\v)\int_0^\tau\|\psi_y\|^2d\tau,
\end{array}\label{(3.28)}
\end{equation}
where in the third inequality we have used the fact that
$\|V_y\|_{L^\i}\leq C\v^{\f12}$ because of  \eqref{(3.24)}.

Similarly, we have
\begin{equation}
\begin{array}{ll}
\di \int_0^\tau\int_{\mathbf
R}\nu\frac{v_y}{v^2}\zeta_y\zeta_{yy}dyd\tau\\
\di\leq
\b\int_0^\tau\|\zeta_{yy}\|^2d\tau+C_\b(\chi^4+\v)\int_0^\tau\|\zeta_y\|^2d\tau.
\end{array}\label{(3.29)}
\end{equation}
The remaining terms can be estimated directly by using
\eqref{(3.23)} and the fact that
$$
[U_y(\f1v-\f1V)]_y=O(1)[|(U_{yy},U_yV_y)||\phi|+|U_y||\phi_y|],
$$
$$
[\nu\Theta_y(\f1v-\f1V)]_y=O(1)[|(\Theta_{yy},\Theta_yV_y)||\phi|+|\Theta_y||\phi_y|].
$$
Hence,  if we take $\b$ suitably small, then we obtain
\begin{equation}
\begin{array}{ll}
\di
\|(\psi_y,\zeta_y)(\tau,\cdot)\|^2+\int_0^\tau\|(\psi_{yy},\zeta_{yy})\|^2d\tau\\
\di\quad \leq C\int_0^\tau\|(\phi_y,\psi_y,\zeta_y)\|^2d\tau
+C\int_0^\tau(\tau+\tau_0)^{-2}\|(\phi,\zeta)\|^2d\tau+C\v^{\f75}\\
\di\quad +C\d^{CD}\int_0^\tau\int_{\bf R}\v(1+\v\tau)^{-1}e^{-\f{C\v
y^2}{1+\v\tau}}|(\p,\z)|^2dy d\tau.
\end{array}\label{(3.30)}
\end{equation}
The combination of \eqref{(3.18)}, \eqref{(3.25)} and \eqref{(3.30)}
yields that
\begin{equation}
\begin{array}{ll}
\di
\|(\phi,\psi,\zeta)(\tau,\cdot)\|_1^2+\int_0^\tau\Big[\|\phi_y\|^2+\|(\psi_{y},\zeta_{y})\|_1^2\Big]d\tau\\
\di\quad \leq
C\int_0^\tau(\tau+\tau_0)^{-2}\|(\phi,\zeta)\|^2d\tau+C\v^{\f25}\\
\di\quad +C\d^{CD}\int_0^\tau\int_{\bf
R}\v(1+\v\tau)^{-1}e^{-\f{C_0\v y^2}{1+\v\tau}}|(\p,\z)|^2dy d\tau.
\end{array}\label{(3.31)}
\end{equation}

In order to close the estimate, we only need to control the last
term in \eqref{(3.31)}, which comes from the viscous contact wave.
For this,  we will apply the following technique by using the heat
kernel motivated by \cite{Huang-Li-Matsumura}.

%\vskip 2mm

\begin{lemma} \label{Lemma 3.2} Suppose that $h(\tau,y)$ satisfies
$$
h\in L^\i(0,+\i; L^2(\mathbf{R})),~~h_y\in L^2(0,+\i;
L^2(\mathbf{R})),~~h_\tau\in L^2(0,+\i; H^{-1}(\mathbf{R})),
$$
Then
\begin{equation}
\begin{array}{ll}
&\di \int_0^\tau
\int_{\mathbf{R}^+}\v(1+\v\tau)^{-1}e^{-\f{2a\v y^2}{1+\v\tau}}h^2(\tau,y)dy
d\tau\\[2mm]
&\di \leq C_a\bigg[ \|h(0,y)\|^2+\int_0^\tau \|h_y\|^2
d\tau+\int_0^\tau\langle h_\tau,hg_a^2\rangle_{H^{-1}\times
H^{1}}d\tau\bigg]
\end{array}\label{(3.32)}
\end{equation}
where
$$
g_a(\tau,y)=\v^\f12(1+\v\tau)^{-\f12}\int^{y}_{-\i} e^{-\f{a\v
\eta^2}{1+\v\tau}}d\eta,
$$
and $a>0$ is the constant to be determined later.
\end{lemma}

The proof of  Lemma \ref{Lemma 3.2} is similar to the one given  in
\cite{Huang-Li-Matsumura}. The only difference here is that we need
to be careful about the parameter
 $\v$ in the estimation. Therefore, we omit its proof for brevity.
 Based on Lemma \ref{Lemma 3.2}, we can
obtain

%\vskip 2mm

\begin{lemma} \label{Lemma 3.3} There exists a constant $C>0$ such that
if $\d^{CD}$ and $\v_0$ are small enough, then we have
\begin{equation}
\begin{array}{ll}
&\di \int_0^\tau\int_{{\bf R}}\v(1+\v\tau)^{-1}e^{-\f{C_0\v
y^2}{1+\v\tau}} |(\p,\psi,\z)|^2 dy
d\tau\\[3mm]
&\di  \leq
C\|(\p,\psi,\z)(\tau,\cdot)\|^2+C\int_0^\tau\|(\p_y,\psi_y,\z_y)\|^2d\tau\\
&\di +C\int_0^\tau (\tau+\tau_0)^{-\f{3}{2}}\|(\p,\psi,\zeta)\|^2
d\tau+C\v^{\f25}.
\end{array}\label{(3.33)}
\end{equation}
\end{lemma}
\begin{proof}
  From the equation $\eqref{(2.24)}_2$ and the fact that
$p-P=\f{R\z-P\p}{v}$, we have
$$
\psi_\tau+(\f{R\z-P\p}{v})_y=(\f{u_y}{v}-\f{U_y}{V})_y-\v Q_1.
$$
Then
\begin{equation}
(R\z-P\p)_y=\f{R\z-P\p}{v}(V_y+\p_y)-v\psi_\tau+
v(\f{u_y}{v}-\f{U_y}{V})_y-v\v Q_1.\label{(3.34)}
\end{equation}

Let
$$
G_b(\tau,y)=\v(1+\v\tau)^{-1}\int^{y}_{-\i} e^{-\f{b\v
\eta^2}{1+\v\tau}}d\eta,
$$
where $b$ is a positive constant to be determined later. Multiplying
the equation \eqref{(3.34)} by $G_b(R\z-P\p)$ gives
\begin{equation}
\begin{array}{ll}
&\di
\left[\f{G_b(R\z-P\p)^2}{2}\right]_y-(G_b)_y\f{(R\z-P\p)^2}{2}\\[2mm]
=&\di \f{G_b(R\z-P\p)^2}{v}(V_y+\p_y)-G_b v(R\z-P\p)\psi_\tau\\[2mm]
&\di +G_b v(R\z-P\p)(\f{u_y}{v}-\f{U_y}{V})_y-\v G_b v(R\z-P\p)Q_1.
\end{array}\label{(3.35)}
\end{equation}
Note that
\begin{equation}
\begin{array}{ll}
\di -G_b v(R\z-P\p)\psi_\tau&\di= -[G_b v(R\z-P\p)\psi]_\tau+[G_b
v(R\z-P\p)\psi]_y\\[3mm]
&\di\quad +(G_b v)_\tau(R\z-P\p)\psi+G_b v\psi(R\z-P\p)_\tau,
\end{array}\label{(3.36)}
\end{equation}
\begin{equation}
\begin{array}{ll}
\di (R\z-P\p)_\tau\\[1mm]
\di=R\z_\tau-P_\tau\p-P\p_\tau\\[1mm]
\di
=(\g-1)\bigg[-(p-P)(U_y+\psi_y)+(\f{u_y^2}{v}-\f{U_y^2}{V})+\nu(\f{\t_y}{v}-\f{\T_y}{V})_y
-\v Q_2\bigg]\\
\di\quad -\g P\psi_y-P_\tau\p.
\end{array}\label{(3.37)}
\end{equation}
By using the equality
\begin{equation}
\di -G_b v\g P \psi_y \psi =-[\g G_b vP\f{\psi^2}{2}]_y+\g v
P(G_b)_y\f{\psi^2}{2}+\g(vP)_y G_b\f{\psi^2}{2}, \label{(3.38)}
\end{equation}
 we have
\begin{equation}
\v(1+\v\tau)^{-1}e^{-\f{b\v y^2}{1+\v\tau}}[(R\z-P\p)^2+\g
Pv\psi^2]=[G_b v(R\z-P\p)\psi]_\tau+(\cdots)_{y}+Q_4, \label{(3.39)}
\end{equation}
where
\begin{equation}
\begin{array}{ll}
\di Q_4=&\di -(v G_b)_\tau v(R\z-P\p)\psi-\f{\g\psi^2}{2}(Pv)_yG_b+G_b v\psi P_\tau\p \\[2mm]
&\di +(\g-1)G_b
v\psi\left[(p-P)(U_y+\psi_y)-(\f{u_y^2}{v}-\f{U_y^2}{V})+\v Q_2\right]\\[2mm]
&\di+[G_b v(R\z-P\p)]_y(\f{u_y}{v}-\f{U_y}{V})+(\g-1)\nu(G_b v \psi)_y(\f{\t_y}{v}-\f{\T_y}{V})\\[2mm]
&\di -\f{G_b(R\z-P\p)^2}{v}(V_y+\phi_y)+\v G_b v(R\z-P\p)Q_1.
\end{array}
 \label{(3.41)}
\end{equation}
Note that
$$
\|G_b(\tau,\cdot)\|_{L^\i}\leq C_\a \v^\f12(1+\v\tau)^{-\f12}.
$$
Thus, integrating \eqref{(3.39)} over $(0,\tau)\times\mathbf{R}$
gives
\begin{equation}
\begin{array}{ll}
&\di \int_0^\tau\int_{{\bf R}}\v(1+\v\tau)^{-1}e^{-\f{b\v
y^2}{1+\v\tau}}[(R\z-P\p)^2+\psi^2]dy d\tau\\[3mm]
&\di \leq C\|(\p,\psi,\z)(\tau,\cdot)\|^2+C\int_0^\tau
\|(\p_y,\psi_y,\z_y)(\tau,\cdot)\|^2 d\tau\\[2mm]
&\di +C\int_0^\tau
(\tau+\tau_0)^{-\f{3}{2}}\|(\p,\psi,\z)(\tau,\cdot)\|^2d\tau+C\v^{\f75}\\[2mm]
&\di +C\d^{CD}\int_0^\tau\int_{{\bf R}}\v(1+\v\tau)^{-1}e^{-\f{C_0\v
y^2}{1+\v\tau}}|(\p,\z)|^2dy d\tau.
\end{array}\label{(3.42)}
\end{equation}

In order to get the desired estimate stated in Lemma \ref{Lemma
3.3}, set
$$
h=\f{R}{\g-1}\z+P\p
$$
in Lemma \ref{Lemma 3.2}. We only need to compute the last term on
the right hand side of \eqref{(3.32)} for this given function $h$.
From the energy equation $\eqref{(2.24)}_3$, we have
\begin{equation}
\begin{array}{ll}
\di
h_\tau&\di=-(p-P)\psi_y+[P_\tau\p-(p-P)U_y]+\nu(\f{\t_y}{v}-\f{\T_y}{V})_y+(\f{u_y^2}{v}-\f{U_y^2}{V})-\v
Q_2\\
&\di :=\sum_{i=1}^5 H_i.
\end{array}
\label{(3.43)}
\end{equation}
Thus
\begin{equation}
\begin{array}{ll}
\di\int_0^\tau\langle h_\tau,hg_a^2\rangle_{H^1\times H^{-1}}d\tau
=\sum_{i=1}^5\int_0^\tau\int_{\mathbf{R}}hg_a^2 H_idyd\tau.
\end{array}\label{(3.44)}
\end{equation}
By noticing  that
$$
\|g_a(\tau,\cdot)\|_{L^\i}\leq C_a,
$$
we can estimate $\di \int_0^\tau\int_{\mathbf{R}}hg_a^2
H_idyd\tau(i=2,\cdots,6)$ directly. The estimation on\\ $\di
\int_0^\tau\int_{\mathbf{R}}hg_a^2 H_1dyd\tau$ is more subtle.
Firstly, by using the mass equation $\eqref{(2.24)}_1$, we have
$$
\begin{array}{ll}
\di
hg_a^2 H_1&\di =-(p-P)\psi_y hg_a^2\\[2mm]
&\di=-\f{(\g-1)h+\g P\p}{v}hg_a^2\p_\tau\\[2mm]
&\di=-\f{(\g-1)h^2 g_a^2}{v}\p_\tau-\f{\g Phg_a^2}{2v}(\p^2)_\tau\\[2mm]
&\di=-\big[ \f{(\g-1)h^2\p g_a^2}{v}+\f{\g Ph
\p^2g_a^2}{2v}\big]_\tau
+\f{2(\g-1)h^2\p+\g Ph\p^2}{v}g_a(g_a)_\tau\\[3mm]
&\di\quad-\f{2(\g-1)h^2\p+\g Ph\p^2}{2v^2}g_a^2v_\tau+\f{\g
h\p^2g_a^2}{2v}P_\tau+ \big[\f{2(\g-1)\p h}{v}+\f{\g
P\p^2}{2v}\big]g_a^2h_\tau\\
&\di:=\sum_{i=1}^5 J_i.
\end{array}
$$
Now the terms $J_i(i=1,\cdots,4)$ can be estimated directly, cf.
 \cite{Huang-Li-Matsumura}. Here we
only calculate the term $J_5.$ From \eqref{(3.43)}, we have
$$
J_5=\sum_{i=1}^6\big[\f{2(\g-1)\p h}{v}+\f{\g
P\p^2}{2v}\big]g_a^2H_i:=\sum_{i=1}^5J_5^i.
$$
Now $J_5^1$ can be estimated as follows:
$$
\begin{array}{ll}
\di \int_0^\tau\int |J_5^1|dyd\tau &\di \leq
C\int_0^\tau\int|\psi_y||(\p,\z)|^3dyd\tau\\
&\di \leq C\int_0^\tau
\|(\p,\z)\|^2_{L_\i}\|\psi_y\|\|(\p,\z)\|d\tau\\
&\di \leq C\int_0^\tau\|(\p,\z)_y\|\|\psi_y\|\|(\p,\z)\|^2d\tau\\
&\di \leq
C\sup_{[0,\tau]}\|(\p,\z)(\tau,\cdot)\|^2\int_0^\tau\|(\p,\psi,\z)_y\|^2
d\tau\\
&\di \leq C\chi^2\int_0^\tau\|(\p,\psi,\z)_y\|^2 d\tau.
\end{array}
$$
Note that the other terms $J_5^i(i=2,\cdots,5)$ can be estimated
directly, we omit the details for brevity.

Therefore, by taking the constant $a=\f{C_0}{2}$, we obtain
\begin{equation}
\begin{array}{ll}
&\di \int_0^\tau\int_{{\bf R}} \v(1+\v\tau)^{-1}e^{-\f{C_0\v
y^2}{1+\v\tau}} h^2 dy d\tau\\
&\di  \leq
C\|(\p,\psi,\z)(\tau,\cdot)\|^2+C\int_0^\tau\|(\p_y,\psi_y,\z_y)\|^2d\tau+C\int_0^\tau
(\tau+\tau_0)^{-\f{3}{2}}\|(\p,\psi)\|^2 d\tau\\
&\di +C\v^{\f25}+C(\d^{CD}+\chi)\int_0^\tau \int_{{\bf
R}^+}\v(1+\v\tau)^{-1}e^{-\f{C_0\v
y^2}{1+\v\tau}}|(\p,\z)|^2dyd\tau.
\end{array}\label{(3.45)}
\end{equation}
By taking $b=C_0$ in (\ref{(3.42)}) and by combining the estimates
(\ref{(3.42)}) with (\ref{(3.45)}), we  yield the desired estimation
in Lemma \ref{Lemma 3.3} if we choose suitably small positive
 constants $\d^{CD}$, $\v_0$
and $\chi$.
\end{proof}

Now from \eqref{(3.31)} and Lemma \ref{Lemma 3.3}, if the strength
of the contact wave $\d^{CD}$ and the parameter $\chi$ on the a
priori estimate are suitably small, we can get
$$
\begin{array}{ll}
\di \|(\p,\psi,\z)(\tau,\cdot)\|_1^2  &\di +\int_0^\tau\Big[\|\p_y\|^2+\|(\psi_y,\z_y)\|_1^2\Big] d\tau \\
&\di \leq
C\Big[\int_{0}^\tau(\tau+\tau_0)^{-\f{3}{2}}\|(\p,\psi,\z)\|^2
d\tau+\v^{\f25}\Big].
\end{array}
$$
With this, the Gronwall inequality gives
$$
\|(\p,\psi,\z)(\tau,\cdot)\|_1^2+\int_0^\tau\Big[
\|\p_y\|^2+\|(\psi_y,\z_y)\|_1^2\Big] d\tau \leq C\v^{\f25}.
$$
And then  we complete the proof of Theorem \ref{Theorem 3.1} by
Sobolev imbedding.

\section{Proof of Theorem \ref{Theorem 2.2}: Hydrodynamic limit of Boltzmann equation}
\setcounter{equation}{0}

In the last section, we will prove the fluid dynamic limit for the
Boltzmann equation to the Riemann solution for the Euler equations
as stated in Theorem \ref{Theorem 2.2}. Again, the proof is based on
energy estimates for
 the Boltzmann equation
\eqref{(2.52)} in the scaled independent variables. For this, it is
sufficient to prove the following theorem.

\begin{theorem}\label{Theorem 4.1} There exist two small positive
constants $\delta_1$, $\v_1$, and a global\break Maxwellian
$\mb{M}_\star=\mb{M}_{[v_\star,u_\star,\theta_\star]}$ such that if
the initial data and the strength of the contact wave  $\delta^{CD}$
satisfy
\begin{equation}
\mathcal{N}(\tau)|_{\tau=0}+\delta^{CD} \le \delta_1, \label{(4.9)}
\end{equation}
and the Knudsen number  $\v\leq \v_1$, then the problem
(\ref{(2.52)}) admits a unique global solution $f^\v(\tau,y,\xi)$
satisfying
\begin{equation}
\begin{array}{l}
\di
\sup_{\tau,y}\|f^\v(\tau,y,\xi)-\mb{M}_{[V,U,\Theta]}(\tau,y,\xi)
\|_{L^2_{\xi}(\frac{1}{\sqrt{\mb{M}_\star}})}\le
C\v^{\f15}.\\
\end{array}
\label{(4.10)}
\end{equation}
Here, $\mathcal{N}(\tau)$ is defined by \eqref{(4.11)} below.
\end{theorem}

 \begin{remark} If we choose the initial data for the Boltzmann equation
(\ref{(2.52)}) as
\begin{equation}
f^\v(0,y,\xi)=\mb{M}_{[V,U,\T]}(0,y,\xi)=\mb{M}_{[V(0,y),U(0,y),\T(0,y)]}(\xi),
\label{(4.12)}
\end{equation}
then
\begin{equation}
\mathcal{N}(\tau)|_{\tau=0}=O(1)\bigg[\|(\T_y,U_y)\|^2+\|(V_{yy},\T_{yy},U_{yy})\|^2\bigg]\bigg|_{\tau=0}=O(1)\v^{\f12}.
\label{(4.13)}
\end{equation}
In this case, the functional measuring the perturbation
$\mathcal{N}(\tau)$ at $\tau=0$ is smaller than the estimate given
in Theorem \ref{Theorem 4.1} that is of the order of $O(\v^{\f25})$
because $\v$ is small.
\end{remark}

Consider the reformulated system \eqref{(2.45)} and \eqref{(2.51)}.
Since the local existence of solution to \eqref{(2.45)} and
\eqref{(2.51)} is now standard, cf. \cite{Guo} and \cite{UYZ}, to
prove the global existence, we only need to close the following a
priori estimate by the continuity argument:
\begin{equation}
\begin{array}{ll}
\di \mathcal{N}(\tau)=&\di\sup_{0\leq \tau^\prime\leq \tau}\Bigg\{
\|(\p,\psi,\z)(\tau^\prime,\cdot)\|_1^2+\int\int\f{|\mb{G} _1|^2}{\mb{M}_\star}d\xi dy\\
&\di +\sum_{|\a^\prime|=1} \int\int\f{|\partial^{\a^\prime} \mb{G}
|^2}{\mb{M}_\star}d\xi dy+\sum_{|\a|=2}\int\int\f{|\partial^\a
f|^2}{\mb{M}_{\star}}d\xi dy\Bigg\}\leq \chi^2,
\end{array}
\label{(4.11)}
\end{equation}
where  $\partial^\a,\partial^{\a^\prime}$ denote the derivatives
with respect to $y$ and $\tau$ respectively, and $\chi$ is a small
positive constant depending on  the initial data and the strength of
the contact wave, and $\mb{M}_\star$ is a global Maxwellian to be
chosen later.

Note that the a priori assumption (\ref{(4.11)}) implies that
\begin{equation}
\|(\p,\psi,\z)\|^2_{L_\i}\leq C\chi^2, \label{(4.15)}
\end{equation}
and
\begin{equation}
\|\int\f{\mb{G} _1^2}{\mb{M}_\star}d\xi\|_{L_\i^y}\leq
C\left(\int\int\f{\mb{G} _1^2}{\mb{M}_\star}d\xi
dy\right)^{\f{1}{2}}\cdot\left(\int\int\f{|\mb{G}
_{1y}|^2}{\mb{M}_\star}d\xi dy\right)^{\f{1}{2}}\leq C(\v+\chi^2),
\label{(4.16)}
\end{equation}
and for $|\a|=1$,
\begin{equation}
\|\int\f{|\partial^{\a} \mb{G} |^2}{\mb{M}_\star}d\xi\|_{L_\i^y}\leq
C\left(\int\int\f{|\partial^{\a} \mb{G} |^2}{\mb{M}_\star}d\xi
dy\right)^{\f{1}{2}}\cdot\left(\int\int\f{|\partial^{\a} \mb{G}
_y|^2}{\mb{M}_\star}d\xi dy\right)^{\f{1}{2}}\leq
C(\v+\chi^2).\label{(4.17)}
\end{equation}
From (\ref{(1.17)}) and (\ref{(2.41)}), we have
\begin{equation}
\left\{
\begin{array}{l}
\di \p_{\tau}-\psi_{1y}=0,\\
\di \psi_{1\tau}+(p-P)_y
              =-\f{4}{3}(\f{\mu(\T)}{V}U_{1y})_y-\v Q_{1}-\int\xi_1^2\mb{G} _yd\xi,\\
\di \psi_{i\tau}=-(\f{\mu(\T)}{V}U_{iy})_y-\int\xi_1\xi_i\mb{G}_yd\xi,~~i=2,3,\\
\di \z_{\tau}+(pu_{1y}-PU_{1y})=-(\f{\lambda(\T)}{V}\T_y)_y
              -\f{4}{3}\f{\mu(\T)}{V}U_{1y}^2-\v Q_{2}\\
\di \qquad -\v Q_1 U_1-\f{1}{2}\int\xi_1|\xi|^2\mb{G}
_yd\xi+\sum_{i=1}^3u_i\int\xi_1\xi_i\mb{G} _yd\xi.
\end{array}
\right. \label{(4.18)}
\end{equation}
Thus
\begin{equation}
\|(\p_{\tau},\psi_{\tau},\z_{\tau})\|^2\leq C(\v+\chi^2).
\label{(4.19)}
\end{equation}
Hence, we have
\begin{equation}
\|(v_{\tau},u_{\tau},\t_{\tau})\|^2\leq
C\|(\p_{\tau},\psi_{\tau},\z_{\tau})\|^2+C\|(V_{\tau},U_{\tau},\T_{\tau})\|^2\leq
C(\v+\chi^2). \label{(4.20)}
\end{equation}
In addition,  (\ref{(4.11)}) also implies that
\begin{equation}
\|(v_y,u_y,\t_y)\|^2\leq
C\|(\p_y,\psi_y,\z_y)\|^2+C\|(V_y,U_y,\T_y)\|^2\leq C(\v+\chi^2).
\label{(4.21)}
\end{equation}
Since
\begin{equation}
\|\partial^\a\left(\rho,\rho u,\rho(E+\f{|u|^2}{2})\right)\|^2 \leq
C\int\int\f{|\partial^\a f|^2}{\mb{M}_\star}d\xi dy\le C\chi^2,
\label{(4.22)}
\end{equation}
the inequalities (\ref{(4.20)})-(\ref{(4.22)}) give
\begin{equation}
\begin{array}{ll}
\di\|\partial^\a(v,u,\t)\|^2&\di\leq C\|\partial^\a\left(\rho,\rho
u,\rho(E+\f{|u|^2}{2})\right)\|^2\\
&\di\quad\quad +C\sum_{|\a|=1}\int|\partial^{\a}\left(\rho,\rho
u,\rho(E+\f{|u|^2}{2})\right)|^4dy\\
 &\di\leq C(\v+\chi^2).
\end{array}
\label{(4.23)}
\end{equation}
Thus, for $|\a|=2$, we have
\begin{equation}
\|\partial^\a(\p,\psi,\z)\|^2\leq
C(\|\partial^\a(v,u,\t)\|^2+\|\partial^\a(V,U,\T)\|^2)\leq
C(\v+\chi^2). \label{(4.24)}
\end{equation}
Finally, from the fact that $f=\mb{M}+\mb{G} $, we can obtain for
$|\a|=2$,
\begin{equation}
\begin{array}{l}
\di \int\int\f{|\partial^\a\mb{G} |^2}{\mb{M}_\star}d\xi dy\leq
C\int\int\f{|\partial^\a f|^2}{\mb{M}_\star}d\xi dy+C
\int\int\f{|\partial^\a\mb{M}|^2}{\mb{M}_\star}d\xi
dy\\
\di\quad \leq C\int\int\f{|\partial^\a f|^2}{\mb{M}_\star}d\xi
dy+C\|\partial^\a(v,u,\t)\|^2+C\sum_{|\a^\prime|=1}\int|\partial^{\a^\prime}(v,u,\t)|^4dy\\
\quad \leq C(\v+\chi^2).
\end{array}
\label{(4.25)}
\end{equation}
Before proving the a priori estimate (\ref{(4.11)}), we list some
basic lemmas based on the celebrated H-theorem for later use. The
first lemma is from \cite{GPS}.

 \begin{lemma}\label{Lemma 4.1} There exists a positive
constant $C$ such that
$$
\int\f{\nu(|\xi|)^{-1}Q(f,g)^2}{\tilde{\mb{M}}}d\xi\le
C\left\{\int\f{\nu(|\xi|)f^2}{\tilde{\mb{M}}}d\xi\cdot\int\f{g^2}{\tilde{\mb{M}}}d\xi+
\int\f{f^2}{\tilde{\mb{M}}}d\xi\cdot\int\f{\nu(|\xi|)g^2}{\tilde{\mb{M}}}d\xi\right\},
$$
where $\tilde{\mb{M}}$ can be any Maxwellian so that the above
integrals are well defined.
\end{lemma}

Based on Lemma \ref{Lemma 4.1}, the following three lemmas are taken
from \cite{Liu-Yang-Yu-Zhao}. And the proofs are straightforward by
using Cauchy inequality.

\begin{lemma}\label{Lemma 4.2} If $\t/2<\t_\star<\t$, then there exist two
positive constants $\sigma=\sigma(v,u,\t;\break
v_\star,u_\star,\t_\star)$ and
$\eta_0=\eta_0(v,u,\t;v_\star,u_\star,\t_\star)$ such that if
$|v-v_\star|+|u-u_\star|+|\t-\t_\star|<\eta_0$, we have for
$h(\xi)\in  \mathfrak{N}^\bot$,
$$
-\int\f{h\mb{L}_\mb{M}h}{\mb{M}_\star}d\xi\geq
\sigma\int\f{\nu(|\xi|)h^2}{\mb{M}_\star}d\xi.
$$
\end{lemma}

 \begin{lemma}\label{Lemma 4.3} Under the assumptions in Lemma \ref{Lemma 4.2}, we
have  for each $h(\xi)\in  \mathfrak{N}^\bot$,
$$
\left\{
\begin{array}{l}
\di \int\f{\nu(|\xi|)}{\mb{M}}|\mb{L}_\mb{M}^{-1}h|^2d\xi
\leq \sigma^{-2}\int\f{\nu(|\xi|)^{-1}h^2}{\mb{M}}d\xi,\\
\di \int\f{\nu(|\xi|)}{\mb{M}_\star}|\mb{L}_\mb{M}^{-1}h|^2d\xi\le
\sigma^{-2}\int\f{\nu(|\xi|)^{-1}h^2}{\mb{M}_\star}d\xi.
\end{array}
\right.
$$
\end{lemma}

\begin{lemma}\label{Lemma 4.4} Under the conditions in Lemma \ref{Lemma 4.2},  for any
positive constants $k$ and $\lambda$, it holds that
$$
|\int\f{g_1\mb{P}_1(|\xi|^kg_2)}{\mb{M}_\star}d\xi-\int\f{g_1|\xi|^kg_2}{\mb{M}_\star}d\xi|\le
C_{k,\lambda}\int\f{\lambda|g_1|^2+\lambda^{-1}|g_2|^2}{\mb{M}_\star}d\xi,
$$
where the constant $C_{k,\lambda}$ depends on $k$ and $\lambda$.
\end{lemma}

With the above preparation, we are ready to perform the energy
estimation as follows. Firstly, similar to \eqref{(3.11)}, we can
get
\begin{equation}
\begin{array}{l}
\di\left(\sum_{i=1}^3\f12\psi_i^2+R\T\Phi(\f{v}{V})+\T
\Phi(\f{\t}{\T})\right)_{\tau} +\f{4}{3}\f{\mu(\t)}{v}\psi_{1y}^2
+\sum_{i=2}^3\f{\mu(\t)}{v}\psi_{iy}^2+\f{\lambda(\t)}{v\t}\z_y^2\\
\di+P(U^{R_1}_{1y}+U^{R_3}_{1y})\bigg[\Phi(\f{\t
V}{v\T})+\f53\Phi(\f vV)\bigg]=-PU^{CD}_{1y}\bigg[\Phi(\f{\t
V}{v\T})+\f53\Phi(\f
vV)\bigg]\\
\di+\bigg[(\f{\l(\T)\T_y}{V})_y+\f43\f{\mu(\T)U_{1y}^2}{V}+\v
Q_2\bigg]\bigg[\f23\Phi(\f{v}{V})-\Phi(\f\T\t)\bigg]-\f{4}{3}(\f{\mu(\t)}{v}-\f{\mu(\T)}{V})
U_{1y}\psi_{1y} \\
\di -\f{\z_y}{\t}(\f{\lambda(\t)}{v}
-\f{\lambda(\T)}{V})\T_y+\f{\z\t_y}{\t^2}(\f{\lambda(\t)\t_y}{v}-\f{\lambda(\T)\T_y}{V})
+\f{4\z}{3\t}(\f{\mu(\t)}{v}u_{1y}^2-\f{\mu(\T)}{V}U_{1y}^2)\\[0.3cm]
\di+\f{\z}{\t}\sum_{i=2}^3\f{\mu(\t)}{v}u_{iy}^2-\f{\z}{\t}(\v
Q_2-\v Q_1U_1)-\v Q_{1}\psi_1+N_1+(\cdots)_y,
\end{array}
\label{(4.26)}
\end{equation}
where
\begin{equation}
N_1=-\sum_{i=1}^3\psi_i\int\xi_1\xi_i\Pi_{1y}d\xi+
\f{\z}{\t}(\sum_{i=1}^3u_i\int\xi_1\xi_i\Pi_{1y}d\xi-\f12\int\xi_1|\xi|^2
\Pi_{1y}d\xi). \label{(4.27)}
\end{equation}
The estimation on  the macroscopic terms in \eqref{(4.26)} is almost
same as \eqref{(3.18)} for the compressible Navier-Stokes equations
so that we have
\begin{equation}
\begin{array}{ll}
&\di \|(\phi,\psi,\zeta)(\tau,\cdot)\|^2+\int_0^\tau\Big[\|(\psi_y,\zeta_y)\|^2+\|\sqrt{(U^{R_1}_{1y},U^{R_3}_{1y})}(\phi,\zeta)\|^2\Big]d\tau\\
&\di \leq
C\int_0^\tau(\tau+\tau_0)^{-2}\|(\phi,\zeta)\|^2d\tau+C\v^{\f25}\\
&\di \quad
+C\d^{CD}\v\int_0^\tau\int_{\mathbf{R}}(1+\v\tau)^{-1}e^{-\f{C_0\v
y^2}{1+\v\tau}}(\phi^2+\zeta^2)dyd\tau+\int_0^\tau\int N_1 dyd\tau.
\end{array}
 \label{(4.28)}
\end{equation}
Now we estimate the microscopic term $\di \int_0^\tau\int
N_1dyd\tau$ in \eqref{(4.28)}. For this,  we only estimate the term
$\di T_1=:-\int_0^\tau\int \psi_1\int\xi_1^2\Pi_{1y}d\xi dyd\tau$
because other terms in $\di \int_0^\tau\int N_1dyd\tau$ can be
estimated similarly.

For $T_1$, integration by parts with respect to $y$ and Cauchy
inequality yield
\begin{equation}
\begin{array}{ll}
T_1&\di =\int_0^\tau\int\psi_{1y}\int\xi_1^2\Pi_{1}d\xi dyd\tau\\
&\di \leq \b
\int_0^\tau\|\psi_{1y}\|^2d\tau+C_\b\int_0^\tau\int|\int\xi_1^2\Pi_1d\xi|^2dyd\tau.
\end{array}\label{(4.29)}
\end{equation}
By \eqref{(2.49)}, we have
\begin{equation}
\begin{array}{ll}
\di  \int_0^\tau\int|\int\xi_1^2\Pi_1d\xi|^2dyd\tau\\
 \di \leq C
\int_0^\tau\int|\int\xi_1^2\mb{L}_\mb{M}^{-1}(\mb{G}_{\tau})d\xi|^2dyd\tau+C\int_0^\tau\int|\int\xi_1^2\mb{L}_\mb{M}^{-1}(\f{u_1}v\mb{G}
_y)d\xi|^2dyd\tau\\
\di+C\int_0^\tau\int|\int\xi_1^2\mb{L}_\mb{M}^{-1}[\f{1}{v}\mb{P}_1(\xi_1\mb{G}
_y)]d\xi|^2dyd\tau
+C\int_0^\tau\int|\int\xi_1^2\mb{L}_\mb{M}^{-1}[Q(\mb{G} ,\mb{G} )]d\xi|^2dyd\tau\\
\di :=\sum_{i=1}^4 T_1^i.
\end{array}\label{(4.30)}
\end{equation}
Let $\mb{M}_\star$ be a global Maxwellian with its state
$(v_\star,u_\star,\t_\star)$ satisfying $\f12\t<\t_\star<\t$ and
$|v-v_\star|+|u-u_\star|+|\t-\t_\star|\le \eta_0$ so that Lemma
\ref{Lemma 4.2} holds. Then we can obtain
\begin{equation}
\begin{array}{ll}
T_1^1&\di \leq
C\int_0^\tau\int|\int\f{\nu(|\x|)|\mb{L}_\mb{M}^{-1}\mb{G}_{\tau}|^2}{\mb{M}_\star}
d\x\cdot \int \nu^{-1}(|\x|)\x_1^4\mb{M}_\star d\x| dyd\tau\\
&\di\leq C\int_0^\tau\int\int
\f{\nu^{-1}(|\x|)|\mb{G}_{\tau}|^2}{\mb{M}_\star}d\x dyd\tau.
\end{array}\label{(4.31)}
\end{equation}
Similarly,
\begin{equation}
T_1^2 \leq C\int_0^\tau\int\int
\f{\nu^{-1}(|\x|)|\mb{G}_{y}|^2}{\mb{M}_\star}d\x
dyd\tau.\label{(4.32)}
\end{equation}
Moreover,
\begin{equation}
\begin{array}{ll}
\di T_1^3 \leq
C\int_0^\tau\int|\int\f{\nu(|\x|)|\mb{L}_\mb{M}^{-1}[\f{1}{v}\mb{P}_1(\xi_1\mb{G}
_y)]|^2}{\mb{M}_{[2v_\star,2u_\star,2\t_\star]}} d\x\cdot \int
\nu^{-1}(|\x|)\x_1^4\mb{M}_{[2v_\star,2u_\star,2\t_\star]}d\x|
dyd\tau\\
\quad\di \leq C\int_0^\tau\int\int
\f{\nu^{-1}(|\x|)|\f{1}{v}\mb{P}_1(\xi_1\mb{G}
_y)|^2}{\mb{M}_{[2v_\star,2u_\star,2\t_\star]}}d\x dyd\tau\\
\di\quad  \leq C\int_0^\tau\int\int \f{\nu^{-1}(|\x|)|\mb{G}
_y|^2}{\mb{M}_\star}d\x dyd\tau.
\end{array}
\label{(4.33)}
\end{equation}
From Lemma \ref{Lemma 4.1}, we have
\begin{equation}
\begin{array}{ll}
T_1^4 &\di \leq C\int_0^\tau\int\int\f{\nu^{-1}(|\x|)|Q(\mb{G}
,\mb{G}
)|^2}{\mb{M}_\star}d\xi dyd\tau \\
&\di \leq C\int_0^\tau\int \int\f{\nu(|\x|)|\mb{G}|^2}{\mb{M}_\star}
d\x\cdot \int \f{|\mb{G}|^2}{\mb{M}_\star}d\x dyd\tau\\
&\di\leq C\int_0^\tau\int
\int\f{\nu(|\x|)(|\mb{G}_0|^2+|\mb{G}_1|^2)}{\mb{M}_\star}
d\x\cdot \int \f{|\mb{G}_0|^2+|\mb{G}_1|^2}{\mb{M}_\star}d\x dyd\tau\\
% &\di\leq C(\sup_{\tau,y} \int
%\f{|\mb{G}_1|^2}{\mb{M}_\star}d\x+\v)\int_0^\tau\int\int\f{\nu(|\x|)|\mb{G}_1|^2}{\mb{M}_\star}d\x
% dyd\tau+C\int_0^\tau\|(\T_y,U_y)\|_{L^4(dy)}^4d\tau\\
 &\di\leq C(\v+\chi^2)\int_0^\tau\int\int\f{\nu(|\x|)|\mb{G}_1|^2}{\mb{M}_\star}d\x
dyd\tau+C\v^{\f12}.
\end{array}\label{(4.34)}
\end{equation}
Substituting \eqref{(4.29)}-\eqref{(4.34)} into \eqref{(4.28)}
yields that
\begin{equation}
\begin{array}{ll}
&\di \|(\phi,\psi,\zeta)(\tau,\cdot)\|^2+\int_0^\tau\Big[\|(\psi_y,\zeta_y)\|^2+\|\sqrt{(U^{R_1}_{1y},U^{R_3}_{1y})}(\phi,\zeta)\|^2\Big]d\tau\\
&\di \leq
C\int_0^\tau(\tau+\tau_0)^{-2}\|(\phi,\zeta)\|^2d\tau+C\v^{\f25}\\
&\di \quad
+C\d^{CD}\v\int_0^\tau\int_{\mathbf{R}}(1+\v\tau)^{-1}e^{-\f{C_0\v
y^2}{1+\v\tau}}(\phi^2+\zeta^2)dyd\tau\\
&\di\quad +C\sum_{|\a^\prime|=1}\int_0^\tau\int\int
\f{\nu^{-1}(|\x|)|\partial^{\a^\prime}\mb{G}|^2}{\mb{M}_\star}d\x
dyd\tau\\
&\di\quad
+C(\chi^2+\v)\int_0^\tau\int\int\f{\nu(|\x|)|\mb{G}_1|^2}{\mb{M}_\star}d\x
dyd\tau.
\end{array}
 \label{(4.35)}
\end{equation}

To recover the term $\|\p_y\|^2$ in the integral $\di
\int_0^\tau\cdots d\tau$ in (4.26), as in the previous subsection
for the compressible Navier-Stokes equations, we firstly rewrite the
equation $(\ref{(2.47)})_2$ as
\begin{equation}
\begin{array}{l}
\quad\di\f{4}{3}\f{\mu(\T)}{V}\p_{y\tau}-\psi_{1\tau}-(p-P)_y\\
\di= -\f{4}{3}(\f{\mu({\T})}{V})_y\psi_{1y}
-\f{4}{3}[(\f{\mu({\t})}{v}-\f{\mu({\T})}{V})u_{1y}]_y +\v
Q_{1}+\int\xi_1^2\Pi_{1y}d\xi,
\end{array}
\label{(4.36)}
\end{equation}
by using the equation of conservation of the mass
$(\ref{(2.47)})_1$.

Since
$$
-(p-P)_y=\f{P}{V}\p_y-\f{2}{3V}\z_y+(\f{p}{v}-\f{P}{V})v_y-\f23(\f{1}{v}-\f{1}{V})\t_y,
$$
and
$$
\p_y\psi_{1\tau}=(\p_y\psi_1)_{\tau}-(\p_{\tau}\psi_1)_y+\psi_{1y}^2,
$$
 by multiplying (\ref{(4.36)}) by $\p_y$, we get
\begin{equation}
\begin{array}{l}
\di(\f{2\mu(\T)}{3V}\p_{y}^2-\p_y\psi_1)_{\tau} +\f{P}{V}\p_y^2=
(\f{2\mu(\T)}{3V})_{\tau}\p_{y}^2
+\psi_{1y}^2+\f{2}{3V}\z_y\p_y\\
\quad\di-(\f{p}{v}-\f{P}{V})v_y\p_y+\f23(\f{1}{v}-\f{1}{V})\t_y\p_y-\f{4}{3}(\f{\mu(\T)}{V})_y\psi_{1y}\p_y\\
\quad\di -\f{4}{3}[(\f{\mu({\t})}{v}-\f{\mu({\T})}
{V})u_{1y}]_y\p_y+\v Q_1\p_y+\int\xi_1^2\Pi_{1y}d\xi\p_y.
\end{array}
\label{(4.37)}
\end{equation}
Integrating (\ref{(4.37)}) with respect to $\tau,y$ and using the
Cauchy inequality yield
\begin{equation}
\begin{array}{l}
\di \|\p_y(\tau,\cdot)\|^2+\int_0^\tau\|\p_y\|^2d\tau\leq
C\|\psi_1(\tau,\cdot)\|^2+
C\int_0^\tau\|(\psi_y,\z_y)\|^2d\tau\\
\di~~
+C\d^{CD}\v\int_0^\tau\int_{\mathbf{R}}(1+\v\tau)^{-1}e^{-\f{C_0\v
y^2}{1+\v\tau}}|(\phi,\psi,\zeta)|^2dyd\tau+C\int_0^\tau(\tau+\tau_0)^{-2}\|(\phi,\zeta)\|^2d\tau\\
\di
~~+C\v^{\f75}+C\chi\int_0^\tau\|\psi_{1yy}\|^2d\tau+\int_0^\tau\int
|\int\xi_1^2\Pi_{1y}d\xi|^2 dyd\tau.
\end{array}
\label{(4.38)}
\end{equation}
For the microscopic term $\di \int_0^\tau\int
|\int\xi_1^2\Pi_{1y}d\xi|^2 dyd\tau$, by \eqref{(2.50)}, we have
\begin{equation}
\begin{array}{l}
 \di \int_0^\tau\int |\int\xi_1^2\Pi_{1y}d\xi|^2 dyd\tau\\
 \di \leq
C\Big[\int_0^\tau\int|\int\xi_1^2(\mb{L}_\mb{M}^{-1}\mb{G}_{\tau})_y
d\xi|^2
dyd\tau+\int_0^\tau\int|\int\xi_1^2(\mb{L}_\mb{M}^{-1}\f{u_1}v\mb{G}
_y)_y d\xi|^2dyd\tau\\
\di+\int_0^\tau\int|\int\xi_1^2[\mb{L}_\mb{M}^{-1}\f{1}{v}\mb{P}_1(\xi_1\mb{G}
_y)]_yd\xi|^2dyd\tau
+\int_0^\tau\int|\int\xi_1^2[\mb{L}_\mb{M}^{-1}Q(\mb{G}
,\mb{G} )]_yd\xi|^2dyd\tau\Big]\\
\di:=\sum_{i=1}^4T_2^{i}.
\end{array}
\label{(4.39)}
\end{equation}
Note that the inverse of the linearized operator
$\mb{L}_\mb{M}^{-1}$ satisfies that , for any $h\in
\mathcal{N}^\bot$,
\begin{equation}
\begin{array}{l}
(\mb{L}_\mb{M}^{-1}h)_{\tau}=\mb{L}_\mb{M}^{-1}(h_{\tau})-2\mb{L}_\mb{M}^{-1}\{Q(\mb{L}_\mb{M}^{-1}h,\mb{M}_{\tau})\},\\[2mm]
(\mb{L}_\mb{M}^{-1}h)_y=\mb{L}_\mb{M}^{-1}(h_y)-2\mb{L}_\mb{M}^{-1}\{Q(\mb{L}_\mb{M}^{-1}h,\mb{M}_y)\}.
\end{array}
\label{(4.40)}
\end{equation}
Then we have
\begin{equation}
\begin{array}{ll}
\di T_2^1&\di \leq
C\int_0^\tau\int|\int\xi_1^2\mb{L}_\mb{M}^{-1}\mb{G}_{y\tau} d\xi|^2
dyd\tau\\
&\di
\qquad+C\int_0^\tau\int|\int\xi_1^2\mb{L}_\mb{M}^{-1}\{Q(\mb{L}_\mb{M}^{-1}\mb{G}_\tau,\mb{M}_y)\}d\xi|^2
dyd\tau\\
&\di \leq C\sum_{|\a|=2}
\int_0^\tau\int\int\f{\nu^{-1}(|\xi|)}{\mb{M}_\star}|\partial^\a
\mb{G} |^2d\xi dyd\tau\\
&\di\qquad +C\int_0^\tau\int
\int\f{\nu(|\xi|)|\mb{G}_\tau|^2}{\mb{M}_\star}
d\xi\int \f{\nu(|\xi|)|\mb{M}_y|^2}{\mb{M}_\star} d\xi dyd\tau\\
&\di \leq C\sum_{|\a|=2}
\int_0^\tau\int\int\f{\nu^{-1}(|\xi|)}{\mb{M}_\star}|\partial^\a
\mb{G} |^2d\xi dyd\tau\\
&\di\qquad +C\int_0^\tau\int|(v_y,u_y,\t_y)|^2
\int\f{\nu(|\xi|)|\mb{G}_\tau|^2}{\mb{M}_\star} d\xi dyd\tau\\
&\di \leq C\sum_{|\a|=2}
\int_0^\tau\int\int\f{\nu^{-1}(|\xi|)}{\mb{M}_\star}|\partial^\a
\mb{G} |^2d\xi dyd\tau\\
&\di\qquad +C(\v+\chi^2)\int_0^\tau\int
\int\f{\nu(|\xi|)|\mb{G}_\tau|^2}{\mb{M}_\star} d\xi dyd\tau.
\end{array}
\label{(4.41)}
\end{equation}
Similar estimates hold for $T_2^i~(i=2,3)$. Moreover,
\begin{equation}
\begin{array}{ll}
\di T_2^4&\di \leq
C\int_0^\tau\int|\int\xi_1^2\mb{L}_\mb{M}^{-1}Q(\mb{G},\mb{G}_{y})
d\xi|^2 dyd\tau\\
&\di\quad +
C\int_0^\tau\int|\int\xi_1^2\mb{L}_\mb{M}^{-1}\{Q(\mb{L}_\mb{M}^{-1}Q(\mb{G},\mb{G}),\mb{M}_y)\}d\xi|^2
dyd\tau\\
&\di \leq
C\int_0^\tau\int\int\f{\nu(|\xi|)|\mb{G}_y|^2}{\mb{M}_\star}d\xi\int
\f{|\mb{G}|^2}{\mb{M}_\star} d\xi dyd\tau\\
&\di\quad +C\int_0^\tau\int
\int\f{\nu(|\xi|)|\mb{G}|^2}{\mb{M}_\star}
d\xi\int \f{|\mb{G}_y|^2}{\mb{M}_\star} d\xi dyd\tau\\
&\di \quad+C\int_0^\tau\int|(v_y,u_y,\t_y)|^2
\int\f{\nu(|\xi|)|\mb{G}|^2}{\mb{M}_*}
d\xi\int\f{|\mb{G}|^2}{\mb{M}_\star} d\xi dyd\tau\\
&\di \leq
C(\chi^2+\v)\int_0^\tau\int\int\f{\nu(|\x|)(|\mb{G}_1|^2+|\mb{G}_y|^2)}{\mb{M}_\star}d\x
dy d\tau.
\end{array}
\label{(4.42)}
\end{equation}
Substituting \eqref{(4.39)}-\eqref{(4.42)} into \eqref{(4.38)} gives
\begin{equation}
\begin{array}{l}
\di \|\p_y(\tau,\cdot)\|^2+\int_0^\tau\|\p_y\|^2d\tau\leq
C\|\psi_1(\tau,\cdot)\|^2+
C\int_0^\tau\|(\psi_y,\z_y)\|^2d\tau\\
\di
+C\d^{CD}\v\int_0^\tau\int_{\mathbf{R}}(1+\v\tau)^{-1}e^{-\f{C_0\v
y^2}{1+\v\tau}}|(\phi,\psi,\zeta)|^2dyd\tau+C\int_0^\tau(\tau+\tau_0)^{-2}\|(\phi,\zeta)\|^2d\tau\\
\di +C\v^{\f25}+ C\sum_{|\a|=2}\int_0^\tau
\int\int\f{\nu^{-1}(|\xi|)}{\mb{M}_\star}|\partial^\a \mb{G} |^2d\xi
dyd\tau+C\chi\int_0^\tau\|\psi_{1yy}\|^2d\tau\\
\di+C(\v+\chi^2)\int_0^\tau\int
\int\f{\nu(|\xi|)(\sum_{|\a^\prime|=1}|\partial^{\a^\prime}\mb{G}|^2+|\mb{G}_1|^2)}{\mb{M}_\star}
d\xi dyd\tau .
\end{array}
\label{(4.43)}
\end{equation}

We now turn to the time derivatives. To estimate
$\|(\p_{\tau},\psi_{\tau},\z_{\tau})\|^2$, we need to use the system
(\ref{(4.18)}). By multiplying $(\ref{(4.18)})_1$ by $\p_{\tau}$,
$(\ref{(4.18)})_2$ by $\psi_{1\tau}$, $(\ref{(4.18)})_3$ by
$\psi_{i\tau}~(i=2,3)$ and $(\ref{(4.18)})_4$ by $\z_{\tau}$
respectively, and adding them together, after integrating with
respect to $\tau$ and $y$, we have
\begin{equation}
\begin{array}{l}
\di
\int_0^\tau\|(\p_{\tau},\psi_{\tau},\z_{\tau})(\tau,\cdot)\|^2d\tau\leq
C\int_0^\tau\|(\p_y,\psi_y,\z_y)\|^2d\tau+C\v^{\f25}\\
\di\qquad
+\int_0^\tau(\tau+\tau_0)^{-2}\|(\phi,\psi,\z)\|^2d\tau+C\int_0^\tau\int\int\f{\nu(|\xi|)}{\mb{M}_\star}|\mb{G}
_y|^2d\xi dyd\tau\\
\qquad \di
+C\d^{CD}\v\int_0^\tau\int_{\mathbf{R}}(1+\v\tau)^{-1}e^{-\f{C_0\v
y^2}{1+\v\tau}}|(\phi,\psi,\zeta)|^2dyd\tau.
\end{array}
\label{(4.45)}
\end{equation}
The microscopic component $\mb{G} _1$ can be estimated by using the
equation (\ref{(2.52)}). Multiplying (\ref{(2.52)}) by $\f{v\mb{G}
_1}{\mb{M}_\star}$ gives
\begin{equation}
\begin{array}{ll}
\di (v\f{\mb{G} _1^2}{2\mb{M}_\star})_{\tau}-\f{v\mb{G}
_1}{\mb{M}_\star}\mb{L}_\mb{M}\mb{G} _1&\di
=v_\tau\f{|\mb{G}_1|^2}{2\mb{M}_\star}+\bigg\{-\f3{2v\t}\mb{P}_1[\xi_1
(\f{|\xi-u|^2}{2\t}\z_y+\xi\cdot\psi_y)\mb{M}]\\
&\di\qquad+\f{u_1}{v}\mb{G} _y-\f1v\mb{P}_1(\xi_1\mb{G} _y)+Q(\mb{G}
,\mb{G} )-\mb{G} _{0\tau}\bigg\} \f{v\mb{G} _1}{\mb{M}_\star}.
\end{array}
\label{(4.46)}
\end{equation}
Integrating (\ref{(4.46)}) with respect to $\tau, \xi$ and $y$ and
using the Cauchy inequality and Lemma \ref{Lemma 4.1}-\ref{Lemma
4.4} yield that
\begin{equation}
\begin{array}{l}
\di\int\int\f{\mb{G} _1^2}{\mb{M}_\star}(\tau,y,\xi)d\xi
dy+\int_0^\tau\int\int\f{\nu(|\xi|)|\mb{G} _1|^2}{\mb{M}_\star}d\xi dyd\tau\\
\leq\di
C\v^\f25+C\sum_{|\a^\prime|=1}\int_0^\tau\|\partial^{\a^\prime}(\p,\psi,\z)\|^2d\tau
+C\int_0^\tau\int\int\f{\nu(|\xi|)}{\mb{M}_\star}|\mb{G} _y|^2d\xi
dyd\tau,
\end{array}
\label{(4.47)}
\end{equation}
where we have used the fact that
$$
\begin{array}{ll}
\di \int\int\f{v\mb{G} _1^2}{\mb{M}_\star}(\tau=0,y,\xi)d\xi
dy&\di=\int\int\f{v\mb{G} _0^2}{\mb{M}_\star}(\tau=0,y,\xi)d\xi dy\\
&\di \leq C\|(\T_y,U_y)(\tau=0,\cdot)\|^2 \leq C\v^{\f12}.
\end{array}
$$

 Next we derive the estimate on
the higher order derivatives. By multiplying $(\ref{(2.46)})_2$ by
$-\psi_{1yy}$, $(\ref{(2.46)})_3$ by $-\psi_{iyy}~(i=2,3)$,
$(\ref{(2.46)})_4$ by $-\z_{yy}$, and adding them together, we
obtain
\begin{equation}
\begin{array}{l}
\di (\sum_{i=1}^3\f{\psi_{iy}^2}{2}+\f{\z_y^2}{2})_{\tau}
+\f{4}{3}\f{\mu(\t)}{v}\psi_{1yy}^2+\sum_{i=2}^3\f{\mu(\t)}{v}\psi_{iyy}^2
+\f{\lambda(\t)}{v}\z_{yy}^2=\\
\di
-\f{4}{3}(\f{\mu(\t)}{v})_y\psi_{1y}\psi_{1yy}-\sum_{i=2}^3(\f{\mu(\t)}{v})_y\psi_{iy}\psi_{iyy}
-(\f{\lambda(\t)}{v})_y\z_y\z_{yy}\\
\di-\f{4}{3}[(\f{\mu(\t)}{v}-\f{\mu(\T)}{V})U_{1y}]_y\psi_{1yy}-[(\f{\lambda(\t)}{v}-\f{\lambda(\T)}{V})\T_y]_y\z_{yy}+(p-P)_y\psi_{1yy}\\
\di +\v Q_{1}\psi_{1yy}+(pu_{1y}-PU_{1y})\z_{yy}-[\f{4}{3}(\f{\mu(\t)}{v}u_{1y}^2-\f{\mu(\T)}{V}U_{1y}^2)\\[0.3cm]
\di+\sum_{i=2}^3\f{\mu(\t)}{v}u_{iy}^2-(\v
Q_2-\v Q_1U_1)]\z_{yy}+\sum_{i=1}^3\psi_{iyy}\int\xi_1\xi_i\Pi_{1y}d\xi\\
\di
-\z_{yy}(\sum_{i=1}^3u_i\int\xi_1\xi_i\Pi_{1y}d\xi-\f12\int\xi_1|\xi|^2\Pi_{1y}d\xi).
\end{array}
\label{(4.48)}
\end{equation}
Integrating (\ref{(4.48)}) with respect to $\tau,y$ and $\x$ yields
\begin{equation}
\begin{array}{l}
\di
\|(\psi_{y},\z_y)(\tau,\cdot)\|^2+\int_0^\tau\|(\psi_{yy},\z_{yy})\|^2d\tau
\\
\di \leq
C\int_0^\tau\|(\p_y,\psi_y,\z_y)\|^2d\tau+C\int_0^\tau(\tau+\tau_0)^{-2}\|(\phi,\psi,\z)\|^2d\tau+C\v^{\f25}\\
\quad\di
+C\d^{CD}\v\int_0^\tau\int_{\mathbf{R}}(1+\v\tau)^{-1}e^{-\f{C\v
y^2}{1+\v\tau}}|(\phi,\psi,\zeta)|^2dyd\tau\\
\quad \di+
C(\v^\f12+\chi)\int_0^\tau\int\int\f{\nu(|\xi|)}{\mb{M}_\star}|\mb{G}
_1|^2d\xi
dyd\tau+C\sum_{|\a|=2}\int_0^\tau\int\int\f{\nu(|\xi|)}{\mb{M}_\star}|\partial^\a
\mb{G} |^2d\xi
dyd\tau\\
\quad\di
+C(\v^\f12+\chi)\sum_{|\a^{\prime}|=1}\int_0^\tau\int\int\f{\nu(|\xi|)}{\mb{M}_\star}|\partial^{\a^\prime}\mb{G}
|^2d\xi dyd\tau.
\end{array}
\label{(4.49)}
\end{equation}

Again, to recover  $\|\p_{yy}\|^2$ in the time integral in (4.39),
 by applying
$\partial_y$ to $(\ref{(2.46)})_2$, we get
\begin{equation}
\psi_{1y\tau}+(p-P)_{yy} =-\f{4}{3}(\f{\mu(\T)}{V}U_{1y})_{yy}-\v
Q_{1y}-\int\xi_1^2\mb{G} _{yy}d\xi. \label{(4.50)}
\end{equation}
Note that
\begin{equation}
(p-P)_{yy}=-\f{p}{v}\p_{yy}+\f{R}{v}\z_{yy}-\f1v(p-P)V_{yy}
-\f{\p}{v}P_{yy}-\f{2v_y}{v}(p-P)_y-\f{2P_y}{v}\p_y. \label{(4.51)}
\end{equation}
Multiplying (\ref{(4.50)}) by $-\p_{yy}$ and using (\ref{(4.51)})
imply
\begin{equation}
\begin{array}{l}
\di
-\int\psi_{1y}\p_{yy}(\tau,y)dy+\int_0^\tau\int\f{p}{2v}\p_{yy}^2dyd\tau\leq
C\int_0^\tau\|(\psi_{1yy},\z_{yy})\|^2d\tau+C\v^{\f25}\\
\quad\di+C\int_0^\tau(\tau+\tau_0)^{-2}\|(\phi,\psi,\z)\|^2d\tau+C(\v^\f12+\chi)\int_0^\tau\|(\p_y,\psi_y,\z_y)\|^2d\tau\\
\quad\di+C
\sum_{|\a|=2}\int_0^\tau\int\int\f{\nu(|\xi|)}{\mb{M}_\star}|\partial^\a
\mb{G} |^2d\xi dyd\tau.
\end{array}
\label{(4.52)}
\end{equation}

To estimate $\|(\p_{y\tau},\psi_{y\tau},\z_{y\tau})\|^2$ and
$\|(\p_{\tau\tau},\psi_{\tau\tau},\z_{\tau\tau})\|^2$, we use the
system (\ref{(4.18)}) again. By applying $\partial_y$ to
(\ref{(4.18)}), and multiplying the four equations of (\ref{(4.18)})
by $\p_{y\tau}$, $\psi_{1y\tau}$, $\psi_{iy\tau}$ $(i= 2,3)$,
$\z_{y\tau}$ respectively, then adding them together and integrating
with respect to $\tau$ and $y$ give
\begin{equation}
\begin{array}{l}
\di \int_0^\tau\|(\p_{y\tau},\psi_{y\tau},\z_{y\tau})\|^2d\tau\leq
 C\int_0^\tau\|(\p_{yy},\psi_{yy},\z_{yy})\|^2d\tau+C\v^{\f25}\\
\quad\di+C\int_0^\tau(\tau+\tau_0)^{-2}\|(\phi,\psi,\z)\|^2d\tau+C(\v^\f12+\chi)\int_0^\tau\|(\p_y,\psi_y,\z_y)\|^2d\tau\\
\quad\di+C\int_0^\tau\int\int\f{\nu(|\xi|)}{\mb{M}_\star}|\mb{G}
_y|^2d\xi
 dyd\tau+C\sum_{|\a|=2}\int_0^\tau\int\int\f{\nu(|\xi|)}{\mb{M}_\star}|\partial^\a
 \mb{G} |^2d\xi dyd\tau.
\end{array}
\label{(4.53)}
\end{equation}
Similarly, we have
\begin{equation}
\begin{array}{l}
\di
\int_0^\tau\|(\p_{\tau\tau},\psi_{\tau\tau},\z_{\tau\tau})\|^2d\tau\leq
 C\int_0^\tau\|(\p_{y\tau},\psi_{y\tau},\z_{y\tau})\|^2d\tau+C\v^{\f25}\\
\quad\di+C\int_0^\tau(\tau+\tau_0)^{-2}\|(\phi,\psi,\z)\|^2d\tau+C(\v^\f12+\chi)\sum_{|\a^\prime|=1}\int_0^\tau\|\partial^{\a^\prime}(\p,\psi,\z)\|^2d\tau\\
\quad\di+C\int_0^\tau\int\int\f{\nu(|\xi|)}{\mb{M}_\star}|\mb{G}
_y|^2d\xi
 dyd\tau+C\sum_{|\a|=2}\int_0^\tau\int\int\f{\nu(|\xi|)}{\mb{M}_\star}|\partial^\a
 \mb{G} |^2d\xi dyd\tau.
\end{array}
\label{(4.54)}
\end{equation}
A suitable linear combination of \eqref{(4.49)} - \eqref{(4.54)}
gives
\begin{equation}
\begin{array}{l}
\di\|(\psi_{y},\z_y,\p_{yy})(\tau,\cdot)\|^2
+\sum_{|\a|=2}\int_0^\tau\|\partial^\a(\p,\psi,\z)\|^2d\tau\\
\di \leq
C\sum_{|\a|=2}\int_0^\tau\int\int\f{\nu(|\xi|)}{\mb{M}_\star}|\partial^\a
\mb{G}|^2d\xi
dyd\tau+C\sum_{|\a^\prime|=1}\int_0^\tau\int\int\f{\nu(|\xi|)}{\mb{M}_\star}|\partial^{\a^\prime}
\mb{G} |^2d\xi
dyd\tau\\
\di\quad+C(\v^\f12+\chi)\int_0^\tau\int\int\f{\nu(|\xi|)}{\mb{M}_\star}|\mb{G}
_1|^2d\xi
dyd\tau+C\int_0^\tau(\tau+\tau_0)^{-2}\|(\phi,\psi,\z)\|^2d\tau
\\
\di\quad
+C(\v^\f12+\chi)\sum_{|\a^\prime|=1}\int_0^\tau\|\partial^{\a^\prime}(\p,\psi,\z)\|^2d\tau
+C\v^{\f25}.
\end{array}
\label{(4.55)}
\end{equation}
To close the a priori estimate, we also need to estimate the
derivatives on the non-fluid component $\mb{G} $, i.e., $\partial^\a
\mb{G}, (|\a|=1,2)$. Applying $\partial_y$ on (\ref{(2.47)}), we
have
\begin{equation}
\begin{array}{l}
\quad\di
\mb{G} _{y\tau}-(\f{u_1}{v}\mb{G} _y)_y+\{\f1v\mb{P}_1(\xi_1\mb{M}_y)\}_y+\{\f1v\mb{P}_1(\xi_1\mb{G} _y)\}_y\\
\di =\mb{L}_\mb{M}\mb{G} _y+2Q(\mb{M}_y,\mb{G} )+2Q(\mb{G} _y,\mb{G}
).
\end{array}
\label{(4.56)}
\end{equation}
Since
$$
\mb{P}_1(\xi_1\mb{M}_y)=\f3{2v\t}\mb{P}_1[\xi_1(\f{|\xi-u|^2}{2\t}\t_y+\xi\cdot
u_y)\mb{M}],
$$
we have
$$
|\{\f1v\mb{P}_1(\xi_1\mb{M}_y)\}_y|\leq
C(v_y^2+u_y^2+\t_y^2+|\t_{yy}|+|u_{yy}|)|\hat{B}(\xi)|\mb{M},
$$
where $\hat{B}(\xi)$ is a polynomial of $\xi$. This yields that
$$
\begin{array}{ll}
\di\int_0^\tau\int\int|\{\f1v\mb{P}_1(\xi_1\mb{M}_y)\}_y\f{\mb{G}
_y}{\mb{M}_\star}|d\xi dyd\tau \leq
\f{\sigma}{8}\int_0^\tau\int\int\f{\nu(|\xi|)}{\mb{M}_\star}|\mb{G}
_y|^2d\xi dyd\tau\\
\di\qquad\qquad+C\int_0^\tau\|(\psi_{yy},\z_{yy})\|^2d\tau
+C(\v^\f12+\chi)\int_0^\tau\|(\p_y,\psi_y,\z_y)\|^2d\tau+C\v^{\f25}.
\end{array}
$$
Thus, multiplying (\ref{(4.56)}) by $\f{v\mb{G} _y}{\mb{M}_\star}$
and using the Cauchy inequality and Lemmas \ref{Lemma
4.1}-\ref{Lemma 4.4} yield
\begin{equation}
\begin{array}{l}
\di\int\int\f{|\mb{G} _y|^2}{2\mb{M}_\star}(\tau,y,\x)d\xi
dy+\int_0^\tau\int\int\f{\nu(|\xi|)}{\mb{M}_\star}|\mb{G} _y|^2d\xi
dyd\tau\\
\di \leq
C(\v^\f12+\chi)\int_0^\tau\int\int\f{\nu(|\xi|)}{\mb{M}_\star}|\mb{G}
_1|^2d\xi
dyd\tau+C(\v^\f12+\chi)\int_0^\tau\|(\p_y,\psi_y,\z_y)\|^2d\tau\\
\quad\di +C\int_0^\tau\int\int\f{\nu(|\xi|)}{\mb{M}_\star}|\mb{G}
_{yy}|^2d\xi
dyd\tau+C\int_0^\tau\|(\p_{yy},\z_{yy})\|^2d\tau+C\v^{\f25}.
\end{array}
\label{(4.57)}
\end{equation}
Similarly,
\begin{equation}
\begin{array}{l}
\di \int\int\f{|\mb{G} _{\tau}|^2}{2\mb{M}_\star}(\tau,y,\xi)d\xi
dy+\int_0^\tau\int\int\f{\nu(|\xi|)}{\mb{M}_\star}|\mb{G}
_{\tau}|^2d\xi dyd\tau\\
\di \leq C\int_0^\tau\int\int\f{\nu(|\xi|)}{\mb{M}_\star}|\mb{G}
_{y\tau}|^2d\xi
dyd\tau+C(\v^\f12+\chi)\int_0^\tau\int\int\f{\nu(|\xi|)}{\mb{M}_\star}|\mb{G}
_1|^2d\xi
dyd\tau\\
\di+C(\v^\f12+\chi)\int_0^\tau\int\int\f{\nu(|\xi|)}{\mb{M}_\star}|\mb{G}
_{y}|^2d\xi
dyd\tau\\
\di+C(\v^\f12+\chi)\sum_{|\a^\prime|=1}\int_0^\tau\|\partial^{\a^\prime}(\p,\psi,\z)\|^2d\tau
+C\int_0^\tau\|(\psi_{y\tau},\z_{y\tau})\|^2d\tau+C\v^{\f25},
\end{array}
\label{(4.58)}
\end{equation}
where we have used the fact that
$$
\begin{array}{ll}
\di \int\int\f{v|\mb{G} _{\tau}|^2}{2\mb{M}_\star}(\tau=0,y,\xi)d\xi
dy&\di =\int\int\f{|\mb{P}_1(\x_1\mb{M}_y)
|^2}{2v\mb{M}_\star}(\tau=0,y,\xi)d\xi dy\\
&\di \leq C\|(v,u,\t)_y(\tau=0,\cdot)\|^2\\
&\di =C\|(V,U,\T)_y(\tau=0,\cdot)\|^2\leq C\v^{\f12}.
\end{array}
$$

 Finally, we estimate the highest
order derivatives, that is, $\int\psi_{1y}\p_{yy}dy$ and\\
$\int_0^\tau\int\int \f{\nu(|\xi|)|\partial^\a \mb{G}
|^2}{\mb{M}_\star}d\xi dyd\tau$ with $|\a|=2$ in (\ref{(4.55)}). To
do so, it is sufficient to study  $\int\int \f{|\partial^\a
f|^2}{\mb{M}_\star}d\xi dy~(|\a|=2)$ in view of (\ref{(4.22)})-
(\ref{(4.25)}). For this, from (\ref{(2.53)}) we have
$$
vf_\tau-u_1f_y+\x_1f_y=vQ(f,f)= v[\mb{L}_\mb{M}\mb{G} +Q(\mb{G}
,\mb{G} )].
$$
Applying $\partial^\a$ $(|\a|=2)$ to the above equation gives
\begin{equation}
\begin{array}{ll}
\di v(\partial^\a f)_{\tau}-v\mb{L}_\mb{M}\partial^\a\mb{G}
-u_1(\partial^\a f)_y+\x_1(\partial^\a
f)_y\\[3mm]
\di =-\partial^\a v f_\tau+\partial^\a u_1
f_y-\sum_{|\a^\prime|=1}[\partial^{\a-\a^\prime}v\partial^{\a^\prime}f_\tau-\partial^{\a-\a^\prime}u_1\partial^{\a^\prime}f_y]\\
\di \quad +[\partial^\a(v\mb{L}_\mb{M}\mb{G}
)-v\mb{L}_\mb{M}\partial^\a\mb{G} ]+\partial^\a[vQ(\mb{G} ,\mb{G}
)].
\end{array}
\label{(4.59)}
\end{equation}
Multiplying (\ref{(4.59)}) by $\f{\partial^\a
f}{\mb{M}_\star}=\f{\partial^\a \mb{M}}{\mb{M}_\star}+\f{\partial^\a
\mb{G} }{\mb{M}_\star}$ yields
\begin{equation}
\begin{array}{l}
\quad\di (\f{v|\partial^\a f|^2}{2\mb{M}_\star})_{\tau}-
v\mb{L}_\mb{M}\partial^\a \mb{G} \cdot \f{\partial^\a \mb{G}
}{\mb{M}_\star}\\ \di =\f{\partial^\a f}{\mb{M}_\star}\bigg\{
-\partial^\a v f_\tau+\partial^\a u_1
f_y-\sum_{|\a^\prime|=1}[\partial^{\a-\a^\prime}v\partial^{\a^\prime}f_\tau-\partial^{\a-\a^\prime}u_1\partial^{\a^\prime}f_y]\\
\di \quad +[\partial^\a(v\mb{L}_\mb{M}\mb{G}
)-v\mb{L}_\mb{M}\partial^\a\mb{G} ]+\partial^\a[vQ(\mb{G} ,\mb{G}
)]\bigg\}+v\mb{L}_\mb{M}\partial^\a \mb{G} \cdot\f{\partial^\a
\mb{M}}{\mb{M}_\star}+(\cdots)_y.
\end{array}
\label{(4.60)}
\end{equation}
Hence,
$$
\begin{array}{l}
\di \int_0^\tau\int\int |\partial^\a v f_\tau\f{\partial^\a
f}{\mb{M}_\star}|d\x dyd\tau
\\ \di \leq \int_0^\tau\int \bigg[|\partial^\a
v|\int(|\mb{M}_\tau|+|\mb{G}
_\tau|)\f{|\partial^\a\mb{M}|+|\partial^\a\mb{G}
|}{\mb{M}_\star}d\x\bigg]
dyd\tau\\
\di \leq
C(\v+\chi^2)\int_0^\tau\|\partial^{\a}(\p,\psi,\z)\|^2d\tau+\f{\sigma}{16}\int_0^\tau\int\int
\f{v|\partial^\a\mb{G} |^2}{\mb{M}_\star} d\x dyd\tau\\
\di \quad +C(\v+\chi^2)\int_0^\tau\int\int \f{|\mb{G}
_\tau|^2}{\mb{M}_\star} d\x
dyd\tau\\
\di\quad
+C(\v^\f12+\chi)\sum_{|\a^\prime|=1}\int_0^\tau\|\partial^{\a^\prime}(\p,\psi,\z)\|^2d\tau+C\v^{\f25},
\end{array}
$$
and
$$
\begin{array}{l}
\di \sum_{|\a^\prime|=1}\int_0^\tau\int\int
|\partial^{\a-\a^{\prime}}v\partial^{\a^\prime}f_\tau \f{\partial^\a
f}{\mb{M}_\star}| d\x dyd\tau\\
\di \leq
\sum_{|\a^\prime|=1}\int_0^\tau\int|\partial^{\a-\a^{\prime}}v|\int
(|\partial^{\a^\prime}\mb{M}_\tau|+|\partial^{\a^\prime}\mb{G}_\tau|)
\f{|\partial^\a
\mb{M}|+|\partial^{\a}\mb{G}|}{\mb{M}_\star} d\x dyd\tau\\
\di \leq \f{\sigma}{16}\int_0^\tau\int\int \f{v|\partial^\a\mb{G}
|^2}{\mb{M}_\star} d\x
dyd\tau+C(\d+\g)\int_0^\tau\|\partial^{\a}(\p,\psi,\z)\|^2d\tau+C\v^{\f25}.
\end{array}
$$
Notice that similar estimates can be obtained for  the terms
$\partial^\a u_1 f_y\f{\partial^\a f}{\mb{M}_\star}$ and\\
$\sum_{|\a^\prime|=1}\partial^{\a-\a^\prime}u_1\partial^{\a^\prime}f_y\f{\partial^\a
f}{\mb{M}_\star}$.

Furthermore, we have
$$
\begin{array}{l}
\di
\partial^\a(v\mb{L}_\mb{M}\mb{G} )-v\mb{L}_\mb{M}\partial^\a\mb{G} =(\partial^\a
v) \mb{L}_\mb{M}\mb{G} +2v Q(\partial^\a\mb{M},\mb{G} )\\
\di ~~
+\sum_{|\a^\prime|=1}\bigg\{2vQ(\partial^{\a-\a^\prime}\mb{M},\partial^{\a^\prime}\mb{G}
) +\partial^{\a-\a^\prime}v[\mb{L}_\mb{M}\partial^{\a^\prime}\mb{G}
+2Q(\partial^{\a^\prime}\mb{M},\mb{G} )]\bigg\},
\end{array}
$$
and
$$
\begin{array}{l}
\di \partial^\a[vQ(\mb{G} ,\mb{G} )]=(\partial^\a
v)Q(\mb{G} ,\mb{G} )+2vQ(\partial^\a\mb{G} ,\mb{G} )\\
\di\qquad
+\sum_{|\a^\prime|=1}\bigg\{vQ(\partial^{\a-\a^\prime}\mb{G}
,\partial^{\a^\prime}\mb{G} )
+2(\partial^{\a-\a^\prime}v)Q(\partial^{\a^\prime}\mb{G} ,\mb{G}
)]\bigg\}.
\end{array}
$$
For illustration, we only estimate
 one of the above terms in the following because the other terms
can be discussed similarly.
$$
\begin{array}{l}
\quad\di\int_0^\tau\int\int\f{v\partial^\a \mb{G} \cdot
Q(\partial^\a \mb{G} ,\mb{G} )}{\mb{M}_\star}d\xi dyd\tau\\\di
\leq\f{\sigma}{16}\int_0^\tau\int \int\f{v|\partial^\a
\mb{G}|^2}{\mb{M}_\star}d\xi dyd\tau\\
\di~+C\int_0^\tau\int\bigg(\int\f{\nu(|\xi|)|\partial^\a \mb{G}
|^2}{\mb{M}_\star}d\xi \cdot\int\f{| \mb{G}
|^2}{\mb{M}_\star}d\xi+\int\f{ |\partial^\a \mb{G}
|^2}{\mb{M}_\star}d\xi \cdot\int\f{\nu(|\xi|)|
\mb{G} |^2}{\mb{M}_\star}d\xi\bigg) dyd\tau\\
\le \di\f{\sigma}{8}
\int_0^\tau\int\int\f{\nu(|\xi|)}{\mb{M}_\star}v|\partial^\a \mb{G}
|^2d\xi dyd\tau\\
\qquad\qquad\qquad\di +C\int_0^\tau|\sup_{y}\int\f{\nu(|\xi|)|
\mb{G} _1|^2}{\mb{M}_\star}d\xi\cdot\int\int \f{|\partial^\a
\mb{G} |^2}{\mb{M}_\star}d\xi dy|d\tau\\
\di\leq  \f{\sigma}{8}
\int_0^\tau\int\int\f{\nu(|\xi|)}{\mb{M}_\star}v|\partial^\a \mb{G}
|^2d\xi dyd\tau\\
\qquad\qquad\qquad\di+C(\v+\chi^2)\int_0^\tau\int\int\f{\nu(|\xi|)[|\mb{G}
_{1y}|^2+|\mb{G} _1|^2]}{\mb{M}_\star}d\xi
dyd\tau\\
\di\leq  \f{\sigma}{8}
\int_0^\tau\int\int\f{\nu(|\xi|)}{\mb{M}_\star}v|\partial^\a \mb{G}
|^2d\xi
dyd\tau+C(\v+\chi^2)\int_0^\tau\|(\p_y,\psi_y,\z_y)\|^2d\tau+C\v^{\f25}\\[3mm]
\di\quad\quad\qquad
+C(\v+\chi^2)\int_0^\tau\int\int\f{\nu(|\xi|)[|\mb{G}
_{y}|^2+|\mb{G} _1|^2]}{\mb{M}_\star}d\xi dyd\tau.
\end{array}
$$
Now we estimate the term $\di \int_0^\tau\int\int
v\mb{L}_\mb{M}\partial^\a \mb{G} \cdot\f{\partial^\a
\mb{M}}{\mb{M}_\star}d\x dyd\tau$ in (\ref{(4.60)}). First, note
that $\mb{P}_1(\partial^\a \mb{M})$ does not contain the term
$\partial^\a(v,u,\t)$ for $|\a|=2$. Thus, we have
\begin{equation}
\begin{array}{l}
\quad\di \int_0^\tau\int\int\f{v\mb{L}_\mb{M}\partial^\a \mb{G}
\cdot\partial^\a \mb{M}}{\mb{M}}d\xi
dyd\tau=\int_0^\tau\int\int\f{v\mb{L}_\mb{M}\partial^\a \mb{G} \cdot
\mb{P}_1(\partial^\a \mb{M})}{\mb{M}}d\xi dyd\tau\\
\di\leq \f{\sigma }{16} \int_0^\tau\int\int\f{v|\partial^\a \mb{G}
|^2}{\mb{M}_\star}d\xi
dyd\tau+C(\v^\f12+\chi)\sum_{|\a^\prime|=1}\int_0^\tau\|\partial^{\a^\prime}(\p,\psi,\z)\|^2d\tau+C\v^{\f25}.
\end{array}
\label{(4.61)}
\end{equation}
Also we can get
\begin{equation}
\begin{array}{l}
\di \int_0^\tau\int\int v\mb{L}_\mb{M}
\partial^\a
\mb{G} \cdot\partial^\a
\mb{M}(\f{1}{\mb{M}_\star}-\f{1}{\mb{M}})d\xi dyd\tau\leq \f{\sigma
}{16}\int_0^\tau\int\int\f{\nu(|\xi|)}{\mb{M}_\star}v|\partial^\a
\mb{G}|^2d\xi
dyd\tau\\
\quad\di+C\eta_0^2~\int_0^\tau\|\partial^\a(\p,\psi,\z)\|^2d\tau+C(\v^\f12+\chi)
\sum_{|\a^\prime|=1}\int_0^\tau\|\partial^{\a^\prime}(\p,\psi,\z)\|^2d\tau+C\v^{\f25},
\end{array}
\label{(4.62)}
\end{equation}
where the small constant $\eta_0$ is defined in Lemma \ref{Lemma
4.2}. The combination of (\ref{(4.61)}) and (\ref{(4.62)}) gives the
estimation on $\di \int_0^\tau\int\int v\mb{L}_\mb{M}\partial^\a
\mb{G} \cdot\f{\partial^\a \mb{M}}{\mb{M}_\star}d\x dyd\tau$.

Thus, integrating (\ref{(4.60)}) and  using the above estimates give
$$
\begin{array}{l}
\di \int\int \f{v|\partial^\a f|^2}{2\mb{M}_\star}(\tau,y,\x)d\xi
dy+\f{\sigma}{2}\int_0^\tau\int\int\f{\nu(|\xi|)}{\mb{M}_\star}v|\partial^\a
\mb{G}|^2d\xi dyd\tau\\
\di\leq
C(\v^\f12+\chi)\sum_{|\a^\prime|=1}\int_0^\tau\|\partial^{\a^\prime}(\p,\psi,\z)\|^2d\tau+
C(\eta_0+\delta+\gamma)\sum_{|\a|=2}\int_0^\tau\|\partial^\a(\p,\psi,\z)\|^2d\tau\\
\quad\di+C(\v^\f12+\chi)\sum_{|\a^\prime|=1}\int_0^\tau\int\int\f{\nu(|\xi|)}{\mb{M}_\star}|\partial^{\a^\prime}
\mb{G} |^2d\xi
dyd\tau+C\v^{\f25}\\
\quad\di+C(\v^\f12+\chi)\int_0^\tau\int\int\f{\nu(|\xi|)}{\mb{M}_\star}|\mb{G}
_1|^2d\xi dyd\tau,
\end{array}
$$
where we have used the fact that
$$
\begin{array}{ll}
&\di \int\int \f{v|\partial^\a f|^2}{2\mb{M}_\star}(\tau=0,y,\x)d\xi
dy =\int\int \f{v|\partial^\a
\mb{M}_{[V,U,\T]}|^2}{2\mb{M}_\star}(\tau=0,y,\x)d\xi dy\\
&\qquad\di \leq
C\|(V,U,\T)_{yy}(\tau=0,\cdot)\|^2+C\|(V,U,\T)_y(\tau=0,\cdot)\|_{L^4}^4\\
&\di\qquad \leq C\v^{\f32}.
\end{array}
$$

Finally, similar to  Lemma \ref{Lemma 3.3} in the previous section,
we can get
\begin{equation*}
\begin{array}{ll}
&\di \int_0^\tau\int_{{\bf R}}\v(1+\v\tau)^{-1}e^{-\f{C_0\v
y^2}{1+\v\tau}} |(\p,\psi,\z)|^2 dy
d\tau\\[3mm]
&\di  \leq
C\|(\p,\psi,\z)(\tau,\cdot)\|^2+C\int_0^\tau\|(\p_y,\psi_y,\z_y)\|^2d\tau+C\v^{\f25}\\[3mm]
&\di +C\int_0^\tau (\tau+\tau_0)^{-\f{4}{3}}\|(\p,\psi,\zeta)\|^2
d\tau+C(\v^{\f12}+\chi)\int_0^\tau\int\int\f{\nu(|\x|)|\mb{G}_1|^2}{\mb{M}_\star}d\x
dyd\tau\\
&\di +C\sum_{|\a^\prime|=1}\int_0^\tau\int\int
\f{\nu^{-1}(|\x|)|\partial^{\a^\prime}\mb{G}|^2}{\mb{M}_\star}d\x
dyd\tau.
\end{array}
\end{equation*}
Note that here we need to estimate the microscopic terms.

In summary, by combining all the above estimates and by choosing the
strength of the contact wave $\d^{CD}$, the bound on the a priori
estimate $\chi$ and the Knudsen number $\v$ to be  suitably small,
we obtain
$$
\begin{array}{ll}
\di
\mathcal{N}(\tau)+\int_0^\tau\Big[\sum_{1\leq|\a|\leq2}\|\partial^\a(\p,\psi,\z)\|^2+\|\sqrt{(U^{R_1}_{1y},U^{R_3}_{1y})}(\phi,\zeta)\|^2\Big]d\tau\\
\di +\int_0^\tau\int\int\f{\nu(|\x|)\mb{G} _1^2}{\mb{M}_\star}d\xi
dyd\tau+\sum_{|\a^\prime|=1}
\int\int\f{\nu(|\x|)|\partial^{\a^\prime} \mb{G}
|^2}{\mb{M}_\star}(\tau,y,\x)d\xi
dy\\
\di +\sum_{|\a|=2}\int\int\f{\nu(|\x|)|\partial^\a
f|^2}{\mb{M}_{\star}}(\tau,y,\x)d\xi dy\leq C\v^{\f25}.
\end{array}
$$
With the energy estimate, we complete the proof of Theorem
\ref{Theorem 4.1}.

\section*{Acknowledgments}The authors would like to thank the
referee for the valuable comments on revision of the paper. The
research of F. M. Huang was supported in part by NSFC Grant No.
10825102 for Outstanding Young scholars, NSFC-NSAF Grant No.
10676037 and 973 project of China, Grant No. 2006CB805902. The
research of Y. Wang was supported by the NSFC Grant No. 10801128.
The research of T. Yang was supported by the General Research Fund
of Hong Kong, CityU \#104310, and the NSFC Grant No. 10871082.

\medskip
Received August 2010; revised October 2010.

\medskip
\end{document}